\documentclass[11pt]{article}

\RequirePackage{amsthm,amsmath,amsfonts,amssymb}
\RequirePackage[numbers,sort&compress]{natbib}
\RequirePackage{hyperref}
\RequirePackage{graphicx}
\usepackage{color}
\usepackage{newlfont}
\usepackage{bbm}
\usepackage{mathtools}

\usepackage{physics}
\usepackage{enumerate}
\usepackage[title]{appendix}
\usepackage{hyperref}
\usepackage{dsfont}
\usepackage{comment}
\usepackage{array}
\usepackage{etoolbox}
\usepackage{enumitem}
\usepackage{cases}
\usepackage{subfigure}
\usepackage{subcaption}

\numberwithin{equation}{section}
\newcounter{countD}

\newtheorem{theorem}{Theorem}[section]
\newtheorem{definition}[theorem]{Definition}

\newtheorem{lemma}[theorem]{Lemma}
\newtheorem{corollary}[theorem]{Corollary}

\newtheorem{remark}[theorem]{Remark}
\newenvironment{Assumption}[1][]%
{\par\medskip\noindent\textbf{Assumption%
\ifx#1\empty\else\ (#1)\fi.}\ \normalfont}%
{\par\medskip}

\begin{document}

\title{Metastability of diffusion processes in narrow tubes}
\author{Wen-Tai Hsu\footnote{Department of Mathematics, University of Maryland, College Park, MD 20742, wthsu@umd.edu}}

\date{\vspace{-5ex}}

\maketitle

\begin{abstract}
We study the metastable behavior of diffusion processes in narrow tube domains, where the metastability is induced by entropic barriers.
We identify a sequence of characteristic time scales $\{T_\epsilon^i\}_{1 \leq i \leq \abs{V'}}$, where $\abs{V'}$ denotes the number of distinct scales among the vertex neighborhoods, and characterize the asymptotic behavior of the diffusion process both at intermediate time scales and at the first critical time scale.
Our analysis relies on a refined understanding of the narrow escape problem in domains with bottlenecks, in particular on estimates for the exit place and on the conditional distribution of the exit time given the exit place, results that may be of independent interest.
\end{abstract}

\section{Introduction}
In this paper, we start with a simple, finite, connected graph $\Gamma = (V,E) \subset \mathbb{R}^d$, with vertices $O_1,\cdots,O_{\abs{V}} \in V$ and edges $I_1,\cdots,I_{\abs{E}} \in E$, where $d=2$ or $3$.
Here, $V$ and $E$ denote the sets of vertices and edges, respectively, and $\abs{V}$ and $\abs{E}$ their cardinalities.
The associated narrow tube domain $G_\epsilon$ consists of the union of $\lambda_k \epsilon$-neighborhoods of the edges $I_k$ and $r_j(\epsilon)$-neighborhoods of the vertices $O_j$, where $\lambda_k$ are positive constants and $r_j(\epsilon) \downarrow 0$ as $\epsilon \downarrow 0$.

We aim to study the diffusion process with reflecting boundary conditions in the domain $G_\epsilon$
\begin{equation}\label{narrow diffusion}
    dZ^\epsilon(t) = \sqrt{2}dB(t) + \nu_\epsilon (Z^\epsilon(t)) d\phi^\epsilon(t), \ Z^\epsilon(0) = z \in G_\epsilon.
\end{equation}
Here, $B(t)$ is a $d$-dimensional Brownian motion, $\nu_\epsilon(z)$ is the unit inward normal vector at the point $z \in \partial G_\epsilon$, and $\phi^\epsilon(t)$ is the local time of $Z^\epsilon(t)$ on the boundary $\partial G_\epsilon$.

Equations of this type were studied in \cite{SDE2} for the case where $r_j(\epsilon)$ is of order $\epsilon$, and in \cite{Hsu25} for more general scaling (see also Chapter 7 in \cite{SDE1}).
Already in \cite{Hsu25}, through the study of the narrow escape problem, it was observed that a metastable behavior arises when
\begin{equation}\label{ass r}
r_j(\epsilon)^d \gg \epsilon^{d-1}, \qquad \text{equivalently,} \qquad \lim_{\epsilon \to 0} \frac{r_j(\epsilon)^d}{\epsilon^{d-1}} = \infty,
\end{equation}
and the narrow tube domain consists of only two vertex neighborhoods connected by a single edge neighborhood.
The purpose of the present work is to describe this metastable behavior more precisely and to extend the analysis to general graph structures.
Accordingly, we will assume condition \eqref{ass r} holds for all $j=1,\cdots,\abs{V}$ throughout the paper.

Metastability refers to the phenomenon in which a stochastic system remains for a long time in a local equilibrium (a metastable state) before making a rare transition to another state.
It is a ubiquitous feature in molecular dynamics, materials science, and many related fields (see, e.g.,\cite{ERV02,BLS15,SS13}).
Various mathematical frameworks have been developed to analyze metastability in random dynamical systems.
These include the large deviation approach (\cite{FW98,FK10p,FK10s}), the potential-theoretic approach (\cite{BEGK04,BGK05,BH15}), the resolvent approach (\cite{LLS24,LLS24+,LMS25,LM25}) and the quasi-stationary approach (\cite{LLN22,LLN25,NETsp}). 
We also mention the general meta-theorems developed in \cite{BL16,FK17,LX16,SMP}.
The references above are not intended to be exhaustive.

Most studies focus on systems with small noise, where metastability arises from an energetic barrier.
In such settings, the deterministic dynamics creates deep potential wells that trap the system for long periods, and rare transitions occur only due to small stochastic perturbations that enable escape from these wells.

In contrast, the metastability studied in this paper originates from an entropic barrier (see, e.g., \cite{K+04,BLS15,NETsp,CDN26}).
Here, the particle does not need to overcome a high energy barrier; instead, it must find a narrow geometric opening through which it can make a rare transition to another metastable state.
This type of mechanism can be viewed as a metastable phenomenon driven by entropic effects.
Because of this fundamentally different mechanism, understanding metastability arising from entropic barriers requires a detailed analysis of geometric effects, which contrasts sharply with the techniques used in the energetic-barrier setting.

In this work, our goal is to understand the asymptotic behavior of the $Z^\epsilon(t(\epsilon))$ as $\epsilon \to 0$, for time scales $t(\epsilon)$ that diverge at different rates. 
To be more precise, let $\Pi:G_\epsilon \to \Gamma$ be the nearest projection.
Our main theorems are the following two:
\begin{theorem}
    There exist transition time scales $\{T_\epsilon^i\}_{i=1}^{\abs{V'}}$ (the number $\abs{V'}$ and the quantities $T_\epsilon^i$ will be defined explicitly later) such that for any $t(\epsilon) > 0$ satisfying $T_{\epsilon}^i \ll t(\epsilon) \ll T_{\epsilon}^{i+1}$ and any continuous function $F \in C(\Gamma)$, we have
	\begin{equation}
		\lim_{\epsilon \to 0} \sup_{x \in \Gamma} \sup_{z: \Pi(z)=x} \abs{ \mathbb{E}_{z} (F(\Pi(Z^\epsilon(t(\epsilon)))) - \sum_{j'=1}^{\abs{V}} F(O_{j'}) \mu^i(x,O_{j'}) } = 0,
	\end{equation}
    where $\mu^i(x,\cdot)$ denotes a limiting distribution of a certain discrete-time Markov chain, whose precise definition will be given later.
\end{theorem}
\begin{theorem}
    Suppose that there exists a vertex $O_{j_1} \in V$ such that $r_{j_1}(\epsilon) \ll r_j(\epsilon)$ for all $j \neq j_1$. Then, for any $s > 0$ and any continuous function $F \in C(\Gamma)$, we have
	\begin{equation}
		\lim_{\epsilon \to 0} \sup_{x \in \Gamma} \sup_{z: \Pi(z)=x}  \abs{ \mathbb{E}_z  F(\Pi(Z^\epsilon(sT_\epsilon^1))) -  \sum_{j=1}^{\abs{V}} p(x,O_j) \mathbf{E}_{O_j} F(Y(s)) } = 0,
	\end{equation}
    where $p(x,O_j)$ and $\{Y(s)\}_{s \geq 0}$ (a continuous-time Markov chain on the space of vertices) will be introduced and defined precisely later.
\end{theorem}
The first theorem establishes the asymptotic behavior of the diffusion process at an intermediate time scale, showing it approximates a discrete-time Markov chain with absorbing states. 
The second theorem then describes the behavior at the first critical time scale, which is shown to resemble a continuous-time Markov chain.
Beyond the first critical time scale, the limiting dynamics involve instantaneous transitions through the faster vertices, and describing them rigorously depends on the graph structure; see Section \ref{sec rem}.
The proofs rely on refined exit estimates for the narrow escape problem in domains with bottlenecks, improving upon the results in \cite{Hsu25}.

The narrow escape problem, which concerns the behavior of Brownian motion within a reflecting domain that features small windows, is a well-established field of study in applied mathematics and physics. 
This problem serves as an essential model for various physical and biological processes, including the propagation of electrical signals along nearly one-dimensional neuronal structures and the diffusion and transport of proteins within cellular networks (see, e.g., \cite{NET,NETBN1,NETnetwork} for applied perspectives).
The rigorous mathematical analysis of this problem has received comparatively less attention. 
Existing mathematical work has mostly focused on fixed smooth domains (see \cite{LPNET,NETsecond,NET4,NETsp,CLNV26,C+26,Cari26,DLNN26}). 
Studies of domains with complex geometric features, such as bottlenecks, which make the domain asymptotically nonsmooth, remain relatively scarce \cite{Hsu25,Cari26}.

In this paper, we establish a more general exit place estimate and show that the exit time, conditioned on the exit edge, is asymptotically exponentially distributed with the same parameter as the unconditional exit time.
In other words, the exit time and the exit edge become asymptotically independent.

Finally, since the long-time behavior of solutions to the associated partial differential equations can be understood through the metastable dynamics of the underlying diffusion process, we use our diffusion asymptotics to characterize the asymptotic behavior of the solution $\rho_\epsilon(t(\epsilon),z)$ to the following PDE,
\begin{equation}
\begin{aligned}
    \begin{dcases}
    \displaystyle
        \frac{\partial \rho_\epsilon }{\partial t} (t,z)
        = \Delta \rho_\epsilon (t,z),
        \   z \in G_\epsilon, \ t \geq 0 \\
        \frac{\partial \rho_\epsilon }{\partial \nu_\epsilon} (t,z) = 0 , \ z \in 	\partial G_\epsilon, \ t \geq 0, \ \ \ \
        \rho_\epsilon(0,z)=\varphi_\epsilon (z), \ z \in G_\epsilon, 
    \end{dcases}
\end{aligned}
\end{equation}
where the initial condition $\varphi_\epsilon$ is assumed to be equicontinuous and equibounded.
This assumption includes the case $\varphi_\epsilon = \varphi|_{G_\epsilon}$ for some $\varphi \in C(\bar{G})$, where $G \coloneqq G_1$.

The paper is organized as follows.
In Section \ref{sec prelim}, we review the construction of the graph $\Gamma$ and the corresponding narrow tube domain $G_\epsilon$, as well as previously established results on exit-time and exit-place estimates.
Section \ref{sec exit} develops estimates for the exit time and exit distribution, while Section \ref{sec exit exp} establishes the exponential law for the exit time.
In Section \ref{sec dis}, we study the asymptotic behavior of the process at the intermediate time scale.
In particular, Section \ref{sec dis markov} shows that, at this time scale, the diffusion process behaves asymptotically as a discrete-time Markov chain, whereas Section \ref{sec pde} investigates the corresponding asymptotic behavior for the PDE.
Finally, Section \ref{sec conti} is devoted to the analysis of the process at the first critical time scale.

\section{Notation and preliminaries}\label{sec prelim}
In this section, we present a brief overview of the notation and recall important results from \cite{Hsu25}.

\subsection{Asymptotic notation}
Here, we define some asymptotic notation that will be used throughout the paper:
\begin{enumerate}
	\item $f(\epsilon) \ll g(\epsilon)$ or $f(\epsilon)=o(g(\epsilon))$ if $\lim_{\epsilon \to 0} \frac{f(\epsilon)}{g(\epsilon)} = 0 $.
	\item $f(\epsilon) \gg g(\epsilon)$ if $\lim_{\epsilon \to 0} \frac{g(\epsilon)}{f(\epsilon)} = 0 $.
	\item $f(\epsilon) \sim g(\epsilon)$ if $\lim_{\epsilon \to 0} \frac{f(\epsilon)}{g(\epsilon)} =1$.
	\item $f(\epsilon)  \asymp g(\epsilon)$ if there exist $c,c'>0$ and $\epsilon_0>0$ such that $c g(\epsilon) \leq f(\epsilon) \leq c' g(\epsilon)$ for all $\epsilon \in (0,\epsilon_0)$.
	\item $f(\epsilon)=O(g(\epsilon))$ or $f(\epsilon) \lesssim g(\epsilon)$ if there exist $c>0$ and $\epsilon_0>0$ such that $\abs{f(\epsilon)} \leq c g(\epsilon)$ for all $\epsilon \in (0,\epsilon_0)$.
\end{enumerate}
In situations involving multiple parameters (e.g., $\epsilon$, $\delta$), we use subscripts to indicate the relevant asymptotic dependencies.
If we only consider the limit $\epsilon \to 0$, we do not add any subscripts.
For instance, we write $f(\epsilon,\delta) \ll_{\epsilon,\delta} g(\epsilon,\delta)$ or $f(\epsilon,\delta) = o_{\epsilon,\delta} (g(\epsilon,\delta))$ to mean $\lim_{\delta \to 0} \lim_{\epsilon \to 0} \frac{f(\epsilon,\delta)}{g(\epsilon,\delta)}=0$.

\subsection{The graph \texorpdfstring{$\Gamma$}{Lg}}\label{sec graph}
Throughout this paper, we consider only a simple, finite, connected graph $\Gamma = (V,E) \subset \mathbb{R}^d$ with $d=2$ or $3$.
The graph consists of vertices $O_1, \cdots, O_{\abs{V}}$ and edges $I_1,\cdots, I_{\abs{E}}$, where $V$ and $E$ denote the sets of vertices and edges, respectively, and $\abs{V}$ and $\abs{E}$ their cardinalities.
Each edge is a straight line segment connecting a pair of distinct vertices. 
See Figure \ref{fig:graph}.
For each $k=1,\cdots,\abs{E}$, we write $I_k \sim O_j$ to indicate that $O_j$ is an endpoint of edge $I_k$.
\begin{figure}[ht]
    \centering
    \subfigure[An example of a simple, finite, connected graph $\Gamma \subset \mathbb{R}^2$.]{\includegraphics[width=0.65\linewidth]{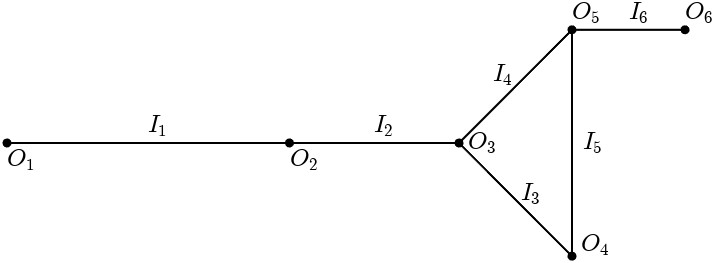}  
    \label{fig:graph}} 
    \subfigure[The corresponding narrow tube domain $G_\epsilon \subset \mathbb{R}^2$.]{\includegraphics[width=0.65\linewidth]{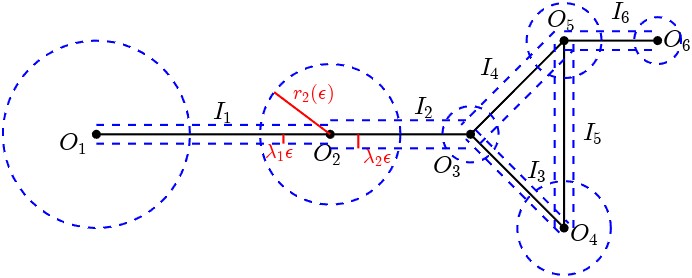}   
    \label{fig:tube}}     
    \vspace*{8pt}
    \caption{Illustration of the graph $\Gamma$ and its associated narrow tube domain $G_\epsilon$.}
\end{figure}

We define a metric $d_\Gamma$ on the graph $\Gamma$ as follows:
if $x,x'$ lie on the same edge, then $d_\Gamma(x,x') \coloneqq d(x,x')$, where $d(\cdot,\cdot) $ denotes the Euclidean distance.
If $x,x'$ lie on different edges, then 
\begin{equation*}
	d_\Gamma(x,x') \coloneqq 
		\min \{ d(x, O_{i_1})+d(O_{i_1}, O_{i_2})+\cdots+d(O_{i_n}, x') \}, 
\end{equation*}
where the minimum is taken over all possible paths from $x$ to $x'$ through a sequence of vertices $O_{i_1}, \cdots, O_{i_n}$ connecting them.
For each $j =1,\cdots,\abs{V}$ and each $k$ such that $I_k \sim O_j$, we define the direction $e_{j,k}$ as the unit vector pointing outward from the vertex $O_j$ along the edge $I_k$.

\subsection{The narrow tube \texorpdfstring{$G_\epsilon$}{Lg}}\label{sec narrow tube}
We consider a domain $G_\epsilon$ consisting of narrow tubes $\Gamma^\epsilon_k$ surrounding the edges $I_k \subset \Gamma$, along with small neighborhoods $\mathcal{E}_j^\epsilon$ around the vertices $O_j \subset \Gamma$. See Figure \ref{fig:tube}. 
More precisely, for each $j =1,\cdots,\abs{V}$, let $\mathcal{E}_j^\epsilon \coloneqq B(O_j, r_j(\epsilon))$ denote the Euclidean ball in $\mathbb{R}^d$ centered at $O_j$ with radius $ r_j(\epsilon)$, where $r_j(\epsilon) > 0$ for all $\epsilon >0$, with $ \lim_{\epsilon \to 0} r_j(\epsilon) = 0$.
For each $k=1,\cdots,\abs{E}$, we consider the edge $I_k$ and define its orthogonal complement $I_k^\perp$ by
\begin{equation*}
	I_k^\perp \coloneqq \{ y \in \mathbb{R}^d : \langle y, e_{j,k} \rangle_{\mathbb{R}^d} = 0 \},
\end{equation*} 
where $e_{j,k}$ is the unit vector pointing outward from a vertex $O_j$ (with $I_k \sim O_j$) along the edge $I_k$.
This definition is well-posed, as the choice of $j$ does not affect $I_k^\perp$; the vectors $e_{j,k}$ corresponding to different endpoints of $I_k$ differ only by a sign.
Additionally, we define the ball $B_k(y,r)$ as
\begin{equation*}
	B_k(y,r) \coloneqq \{ y' \in \mathbb{R}^d: y'-y \in I_k^\perp, \abs{y'-y} \leq r \}.
\end{equation*}
Now let $\Gamma_k^\epsilon \coloneqq I_k + B_k(0,\lambda_k \epsilon)$, where $\lambda_k > 0$ is a constant.
We then define 
\begin{equation*}
    G_\epsilon \coloneqq \bigcup_{j=1}^{\abs{V}} \mathcal{E}_j^\epsilon \cup \bigcup_{k=1}^{\abs{E}} \Gamma_k^\epsilon.
\end{equation*}

Although some results apply when $r_j(\epsilon) \lesssim \epsilon^{(d-1)/d}$, for clarity we impose the following assumption throughout this paper.
\begin{Assumption}
For every $j=1,\cdots,\abs{V}$,
\begin{equation*}
	r_j(\epsilon) \gg \epsilon^{(d-1)/d} .
\end{equation*}
\end{Assumption}

We define the projection $\Pi: G_\epsilon \to \Gamma$ that maps $z \in G_\epsilon$ to the point in $\Gamma$ closest to $z$. 
While this point may not be unique, we can choose it in a unique way. For example, we already have labels on the edges $I_k$. We then assign the projection to the point with the least edge index. 
To be more precise, for every $x \in \Gamma \setminus \bigcup_{j=1}^{\abs{V}} O_j$, let $i(x)$ be the index map such that $i(x)=k$ if and only if $x \in I_k$. 
For each $j=1,\cdots,\abs{V}$, set $i(O_j) = k$ if and only if $O_j \in I_k$ and $k < l$ for all $l$ such that $O_j \in I_l$.
Then, for any $z \in G_\epsilon$, we define $\Pi(z) = x $ if and only if $d(z,x) < d(z,x')$ for all $x' \in \Gamma$, or if $d(z,x) = d(z,x')$ for some $x' \in \Gamma$, but $i(x)<i(x')$ for all such $x' \in \Gamma$.
Note that the projection doesn't depend on $\epsilon$; in fact, one can define the projection in the whole $\mathbb{R}^d$.

The domain $G_\epsilon$ is not smooth because it has corners where the tubular regions meet the vertex neighborhoods.
As in \cite{Hsu25}, we mollify $G_\epsilon$ at scale $\epsilon$, in such a way that
\begin{equation*}
    B(O_j,\rho_jr_j(\epsilon)),\ \Gamma_k^\epsilon\subset G_\epsilon, \quad \text{ for all } j = 1,\cdots,\abs{V}, \ \ k  = 1,\cdots,\abs{E}.
\end{equation*}
Moreover, the boundary of the mollified domain is, in $\epsilon$-scale neighborhoods, the graph of a function whose $W^{1,\infty}$ norm is bounded uniformly in $\epsilon$ and which satisfies a certain monotonicity condition. 
Since the details of this construction are not used in the present paper, we refer to Section 2.3 in \cite{Hsu25} and omit them.

\subsection{Diffusion processes on narrow tubes}\label{sec diff tube prelim}
We consider the diffusion process confined to live on the domain $G_\epsilon$,

\begin{equation}\label{diff eq}
    dZ^\epsilon(t) = \sqrt{2}dB(t) + \nu_\epsilon (Z^\epsilon(t)) d\phi^\epsilon(t), \ Z^\epsilon(0) = z,
\end{equation}
where $z \in G_\epsilon$.
Here, $B(t)$ is a $d$-dimensional Brownian motion defined on a stochastic basis $(\Omega, \mathcal{F}, \{\mathcal{F}_t\}_{t \geq 0}, \mathbb{P})$, $\nu_\epsilon(z)$ is the unit inward normal vector at the point $z \in \partial G_\epsilon$, and $\phi^\epsilon(t)$ is the local time of the process $Z^\epsilon(t)$ on the boundary $\partial G_\epsilon$. 
That is, $\phi^\epsilon(t)$ is an adapted, continuous with probability $1$, non-decreasing process that increases only when $Z^\epsilon(t) \in \partial G$. 
Formally, one can write
\begin{equation*}
    \phi^\epsilon(t) = \int_0^t  \mathbbm{1}_{\{ Z^\epsilon(s) \in \partial G_\epsilon \} } d\phi^\epsilon(s).
\end{equation*}
It can be shown that the generator of $Z^\epsilon$ is the Laplacian $\Delta$, with Neumann boundary conditions on $\partial G_\epsilon$.
We denote the expectation associate with the process $\{Z^\epsilon(t)\}_{t \geq 0}$ starting from $z \in G_\epsilon$ by $\mathbb{E}_{z}$.

\subsection{Previous exit time and exit place estimates}\label{sec prev re}
For each $j=1,\cdots,\abs{V}$, let $L_{j,k} > 0$ for all $k$ such that $I_k \sim O_j$ and 
\begin{equation}
    \vec{L}_j = \bigcup_{k:I_k \sim O_j} \{ (k, L_{j,k}) \}.
\end{equation}
The quantity $L_{j,k}$ represents the length of the bottleneck along the edge $I_k$ adjacent to the vertex $O_j$.

We then define 
\begin{equation*}
\begin{aligned}
    C_{\epsilon,j}^k(L_{j,k}) &\coloneqq  \{ z \in G_\epsilon \cap \Pi^{-1}(I_k) : d_{\Gamma}(\Pi(z),O_j) = L_{j,k} \}, \\
    C_{\epsilon,j}(\vec{L}_j) &\coloneqq \bigcup_{k:I_k \sim O_j} C_{\epsilon,j}^k(L_{j,k}) .
\end{aligned}
\end{equation*}
We also let
\begin{equation*}
\begin{aligned}
\sigma_j^{\epsilon}(\vec{L}_j) &\coloneqq \inf  \{t \geq  0 : Z^\epsilon(t) \in C_{\epsilon,j}(\vec{L}_j) \}.
\end{aligned}
\end{equation*}
For convenience, for every $\delta>0$, we abuse the notation and write
\begin{equation*}
\begin{aligned}
    G_{\epsilon,j}(\delta) &\coloneqq  \{ z \in G_\epsilon : d_{\Gamma}(\Pi(z),O_j) \geq \delta \}, \\
    B_{\epsilon,j}(\delta) &\coloneqq  \{ z \in G_\epsilon : d_{\Gamma}(\Pi(z),O_j) \leq \delta \}, \\
	C_{\epsilon,j}(\delta) &\coloneqq \bigcup_{k:I_k \sim O_j} C_{\epsilon,j}^k(\delta),
\end{aligned}
\end{equation*}
and
\begin{equation*}
\begin{aligned}
\sigma_j^{\epsilon}(\delta) &\coloneqq \inf  \{t \geq  0 : Z^\epsilon(t) \in C_{\epsilon,j}(\delta) \}.
\end{aligned}
\end{equation*}

The following two lemmas were already established in \cite{Hsu25}.
\begin{lemma}[Proposition 3.1 and 3.2 in \cite{Hsu25}]\label{exit place est}
For every $j=1,\cdots,\abs{V}$, and every $k$ such that $I_k \sim O_j$, we have
    \begin{equation*}
        \displaystyle
        \lim_{\delta \to 0}  \lim_{\epsilon \to 0} \sup_{z \in C_{\epsilon,j}(r_j(\epsilon)+3\epsilon)} 
        \abs{ \mathbb{P}_z \left( Z^\epsilon(\sigma_j^{\epsilon}(\delta)) \in C_{\epsilon,j}^k(\delta) \right)  - p_{j,k}} = 0,
    \end{equation*}
    and
    \begin{equation*}
        \displaystyle
           \lim_{\epsilon \to 0} \sup_{z \in C_{\epsilon,j}( r_j(\epsilon)+3\epsilon)} \abs{\mathbb{E}_z \sigma_j^{\epsilon}(\delta) \left( \alpha_j(\epsilon) \delta \right)^{-1} -1 }=0,
    \end{equation*} 
    where
    \begin{equation*}
    		p_{j,k} = \frac{\lambda_k^{d-1}}{\sum_{l:I_l \sim O_j} \lambda_l^{d-1}}, \ \ \ \ \alpha_j(\epsilon) = \frac{r_j(\epsilon)^d V_d}{\sum_{k:I_k \sim O_j} \lambda_k^{d-1} \epsilon^{d-1} V_{d-1}},
    \end{equation*}
    and $V_d = \frac{\pi^{d/2}}{\Gamma(d/2+1)}$ denotes the volume of the unit ball in $\mathbb{R}^d$, where $\Gamma$ is the Gamma function.
\end{lemma}

\section{Exit time and exit place estimates}\label{sec exit}
In this section, our goal is to improve the exit time and exit place estimates.

For each $j=1,\cdots,\abs{V}$, we consider $\vec{L}_j^\epsilon = \bigcup_{k:I_k \sim O_j} \{ (k,L_{j,k}^\epsilon) \}$, which depends weakly on $\epsilon$ in the sense that 
\begin{equation}
    \lim_{\epsilon \to 0} L_{j,k}^\epsilon = L_{j,k} > 0,
\end{equation}
for all $k$ such that $I_k\sim O_j$.
This weak dependence on $\epsilon$ is introduced because, in Section \ref{sec dis} and \ref{sec conti}, we will consider $\sigma^\epsilon(\vec{L}_j^\epsilon)$ starting from $z \in C_{\epsilon,j}(r_j(\epsilon)+3\epsilon)$ with
\begin{equation}
    L_{j,k}^\epsilon = \abs{I_k}-(r_{j'}(\epsilon)+3\epsilon), \ \ \ I_k \sim O_j,O_{j'},
\end{equation}
where $\abs{I_k}$ is the length of the edge $I_k$.
However, it is straightforward to verify that the proofs remain essentially unchanged under this modification.

Since we are concerned only with the local behavior near a given vertex, we omit the subscript $j$ throughout this section for clarity.

Let us first define the sequence of stopping times
\begin{equation*}
\begin{aligned}
\sigma_{n}^{\epsilon,\delta}(\vec{L}^\epsilon) &\coloneqq \inf  \{t \geq  \tau_{n}^{\epsilon,\delta}(\vec{L}^\epsilon) : Z^\epsilon(t) \in  C_{\epsilon}(\vec{L}^\epsilon) \cup C_{\epsilon}(r(\epsilon)+ 3\epsilon ) \}, \\
\tau_{n}^{\epsilon,\delta}(\vec{L}^\epsilon) &\coloneqq \inf  \{t \geq  \sigma_{n-1}^{\epsilon,\delta}(\vec{L}^\epsilon) : Z^\epsilon(t) \in C_{\epsilon}(\delta) \},
\end{aligned}
\end{equation*}
with $\tau_{0}^{\epsilon,\delta}(\vec{L}^\epsilon) \coloneqq 0$.

\begin{remark}
    We emphasize that the sequences of stopping times $\{\sigma_{n}^{\epsilon,\delta}(\vec{L}^\epsilon)\}_{  n \geq 0  }$ and $\{\tau_{n}^{\epsilon,\delta}(\vec{L}^\epsilon)\}_{  n \geq 0  }$ are written without the subscript $j$ for simplicity.
    More precisely, they are $\{\sigma_{j,n}^{\epsilon,\delta}(\vec{L}_j^\epsilon)\}_{  n \geq 0  }$ and $\{\tau_{j,n}^{\epsilon,\delta}(\vec{L}_j^\epsilon)\}_{  n \geq 0  }$, respectively, and should not be confused with $\sigma_j^\epsilon(\vec{L}_j)$.
\end{remark}

\subsection{Exit place estimate}

\begin{lemma}\label{exit place est new}
For every $k:I_k \sim O$, we have
    \begin{equation*}
        \displaystyle
         \lim_{\epsilon \to 0} \sup_{z \in C_{\epsilon}(r(\epsilon)+3\epsilon)} 
        \abs{ \mathbb{P}_z \left( Z^\epsilon(\sigma^{\epsilon}(\vec{L}^\epsilon)) \in C_{\epsilon}^k(L_k^\epsilon) \right)  - p_{k}(\vec{L})} = 0,
    \end{equation*}
    where
    \begin{equation*}
	p_{k}(\vec{L}) = \frac{\lambda_k^{d-1} / L_k}{\sum_{l:I_l \sim O} \lambda_l^{d-1} / L_l}.
\end{equation*}
\end{lemma}
\begin{proof}
For $z \in  C_{\epsilon}(r(\epsilon)+3\epsilon)$, let
\begin{equation*}
	p_{k}^{\epsilon, \delta}(z)= \mathbb{P}_z \left( Z^\epsilon(\sigma^\epsilon(\delta)) \in C_{\epsilon}^k(\delta) \right) = 
	\mathbb{P}_z \left( Z^\epsilon(\tau_{1}^{\epsilon,\delta}(\vec{L}^\epsilon)) \in C_{\epsilon}^k(\delta) \right)
\end{equation*}
and
\begin{equation*}
	p_{k}^{\epsilon, \delta}(z,\vec{L}^\epsilon)=\mathbb{P}_z \left( Z^\epsilon(\sigma_{1}^{\epsilon,\delta}(\vec{L}^\epsilon)) \in C_{\epsilon}^k(L_k^\epsilon) \right);
\end{equation*}
that is, $p_{k}^{\epsilon, \delta}(z,\vec{L}^\epsilon)$ is the probability that the particle escapes through edge $k$ in one cycle.
Then, by Lemma \ref{exit place est} and the strong Markov property,
\begin{equation}\label{total trans prob}
\begin{aligned}
	&p_{k}^{\epsilon, \delta}(z,\vec{L}^\epsilon)  
    = \mathbb{E}_z \left( \mathbbm{1}_{ \{ Z^\epsilon(\tau_{1}^{\epsilon,\delta}(\vec{L}^\epsilon)) \in C_{\epsilon}^k(\delta) \} } \mathbb{P}_{Z^\epsilon(\tau_{1}^{\epsilon,\delta}(\vec{L}^\epsilon))} 
    ( Z^\epsilon(\sigma_0^{\epsilon,\delta}(\vec{L}^\epsilon)) \in C_\epsilon(\vec{L}^\epsilon ) )\right) \\
    &=p_{k}^{\epsilon, \delta}(z) \frac{\delta-r(\epsilon)-3\epsilon}{L_k^\epsilon- r(\epsilon)-3\epsilon} \sim_{\epsilon,\delta} \frac{\lambda_k^{d-1}}{\sum_{l:I_l \sim O} \lambda_l^{d-1}} \frac{\delta}{L_k}
\end{aligned}
\end{equation}
uniformly in $z \in C_\epsilon(r(\epsilon)+3\epsilon)$.
Moreover, by the strong Markov property,
\begin{equation*}
\begin{aligned}
	&\mathbb{P}_z \left( Z^\epsilon(\sigma^{\epsilon}(\vec{L}^\epsilon)) \in C_{\epsilon}^k(L_k^\epsilon) \right)
    = \sum_{i=1}^\infty \mathbb{P}_z \left( Z^\epsilon(\sigma^{\epsilon}(\vec{L}^\epsilon)) \in C_{\epsilon}^k(L_k^\epsilon),  \sigma_i^{\epsilon,\delta}(\vec{L}^\epsilon) = \sigma^{\epsilon}(\vec{L}^\epsilon) \right) \\
	 &=\sum_{i=1}^\infty \mathbb{E}_z \left(  \mathbbm{1}_{ \{ \sigma_{i-1}^{\epsilon,\delta}(\vec{L}^\epsilon) < \sigma^{\epsilon}(\vec{L}^\epsilon)  \} } 
     \mathbb{P}_{\sigma_{i-1}^{\epsilon,\delta}(\vec{L}^\epsilon)}
     \left( Z^\epsilon(\sigma_1^{\epsilon,\delta}(\vec{L}^\epsilon)) \in C_{\epsilon}^k(L_k^\epsilon) \right) \right) \\
     & \sim_{\epsilon,\delta} \sum_{i=1}^\infty \mathbb{P}_z \left(   \sigma_{i-1}^{\epsilon,\delta}(\vec{L}^\epsilon) < \sigma^{\epsilon}(\vec{L}^\epsilon)  \right)  p_{k}^{\epsilon, \delta}(z',\vec{L}^\epsilon) \\
     & \sim_{\epsilon,\delta} \sum_{i=1}^\infty  \left(  1- \sum_{l:I_l \sim O} p_{l}^{\epsilon, \delta}(z'',\vec{L}^\epsilon) \right)^{i-1}  p_{k}^{\epsilon, \delta}(z',\vec{L}^\epsilon) \\
      &\sim_{\epsilon,\delta} \frac{p_{k}^{\epsilon, \delta}(z',\vec{L}^\epsilon)}{\sum_{l:I_l \sim O} p_{l}^{\epsilon, \delta}(z'',\vec{L}^\epsilon)}
	\sim_{\epsilon,\delta} p_{k}(\vec{L})
\end{aligned}
\end{equation*}
uniformly in $z, z',z'' \in C_\epsilon(r(\epsilon)+3\epsilon)$.
\end{proof}

\begin{remark}\label{exit 1 cy}
By \eqref{total trans prob}, we have
\begin{equation*}
\mathbb{P}_z \left( Z^\epsilon(\sigma_{1}^{\epsilon,\delta}(\vec{L}^\epsilon)) \in C_{\epsilon}(\vec{L}^\epsilon) \right) 
=\sum_{k:I_k \sim O}p_{k}^{\epsilon, \delta}(z,\vec{L}^\epsilon)\sim_{\epsilon,\delta} \frac{\sum_{k:I_k \sim O} \lambda_k^{d-1} /L_k}{\sum_{l:I_l \sim O} \lambda_l^{d-1}} \delta
\end{equation*}
uniformly in $z \in C_\epsilon(r(\epsilon)+3\epsilon)$.
\end{remark}

\subsection{Exit time estimate}

\begin{lemma}\label{small exit est 2}
We have
    \begin{equation}
        \displaystyle
           \lim_{\epsilon \to 0} \sup_{z \in C_{\epsilon}( r(\epsilon)+3\epsilon)} \abs{\mathbb{E}_z \sigma^{\epsilon}(\vec{L}^\epsilon) \left( \alpha(\epsilon) \frac{\sum_{l:I_l \sim O} \lambda_l^{d-1}}{\sum_{k:I_k \sim O} \lambda_k^{d-1}/L_k}, \right)^{-1} -1 }=0,
    \end{equation} 
    where 
    \begin{equation}
\alpha(\epsilon) = 
 \frac{r(\epsilon)^d V_{d}}{\sum_{k:I_k \sim O} \lambda_k^{d-1} \epsilon^{d-1} V_{d-1}}
.
\end{equation}

\end{lemma}
\begin{proof}
The idea is to apply Wald’s identity (see, e.g., Lemma 4.4 in \cite{Hsu25} and the proof of Lemma \ref{small exit est 2 new} below), which states that the total exit time equals the time spent in each trajectory multiplied by the expected number of trials.
We decompose $\sigma^\epsilon(\vec{L}^\epsilon)$ into cycles using the stopping times $\tau_n^{\epsilon,\delta}(\vec{L}^\epsilon), \sigma_n^{\epsilon,\delta}(\vec{L}^\epsilon)$.
By Lemma \ref{exit place est}, the dominant term to the time spent in each trajectory is $\alpha(\epsilon) \delta$.
More precisely, each trajectory consists of two parts: $\sigma_i^{\epsilon,\delta}-\tau_i^{\epsilon,\delta}$ and $\tau_i^{\epsilon,\delta}-\sigma_{i-1}^{\epsilon,\delta}$.
By the strong Markov property, the first part is related to a one-dimensional Brownian motion, and the second part corresponds to a narrow escape problem. 
Applying Lemma \ref{exit place est} to the latter shows that it gives the dominant contribution.
In addition, by Remark \ref{exit 1 cy}, the number of trials is 
\begin{equation}
	\left( \mathbb{P}_z \left( Z^\epsilon(\sigma_{1}^{\epsilon,\delta}(\vec{L}^\epsilon)) \in C_{\epsilon}(\vec{L}^\epsilon) \right) \right)^{-1} 
	\sim_{\epsilon,\delta} \frac{\sum_{l:I_l \sim O} \lambda_l^{d-1}}{\delta \sum_{k:I_k \sim O} \lambda_k^{d-1}/L_k}.
\end{equation}
Combining these two estimates yields the desired result.
The detailed proof is omitted.
\end{proof}

Let $\theta_{k}^\epsilon(\vec{L}^\epsilon)$ be the first hitting time on $C_{\epsilon}(\vec{L}^\epsilon)$ conditioned on hitting $C_{\epsilon}^k(L_k^\epsilon)$; that is, 
\begin{equation*}
\theta_{k}^\epsilon(\vec{L}^\epsilon)  =
\frac{\mathbbm{1}_{ \{ Z^\epsilon(\sigma^\epsilon(\vec{L}^\epsilon)) \in C_{\epsilon}^k(L_k^\epsilon)\} }}{
\mathbb{P}_{Z^\epsilon(0)}\left(Z^\epsilon(\sigma^\epsilon(\vec{L}^\epsilon)) \in C_{\epsilon}^k(L_k^\epsilon) \right)} \sigma^\epsilon(\vec{L}^\epsilon).
\end{equation*}

\begin{lemma}\label{small exit est 2 new}
    We have
    \begin{equation}
        \displaystyle
           \lim_{\epsilon \to 0} \sup_{z \in C_{\epsilon}(r(\epsilon)+3\epsilon)} \abs{\mathbb{E}_z \theta_{k}^\epsilon(\vec{L}^\epsilon) \left( \alpha(\epsilon) \frac{\sum_{l:I_l \sim O} \lambda_l^{d-1}}{\sum_{k:I_k \sim O} \lambda_k^{d-1}/L_k} \right)^{-1} -1 }=0,
    \end{equation} 
    where 
    \begin{equation}
\alpha(\epsilon) = 
 \frac{r(\epsilon)^d V_{d}}{\sum_{k:I_k \sim O} \lambda_k^{d-1} \epsilon^{d-1} V_{d-1}}
.
\end{equation}
\end{lemma}
\begin{proof}
We would like to apply Wald’s identity as before.
However, the possible dependence between the exit time and the exit place prevents us from doing so directly.
Let $N^{\epsilon}$ be such that $\sigma_{N^{\epsilon}}^{\epsilon,\delta}(\vec{L}^\epsilon) = \sigma^\epsilon(\vec{L}^\epsilon)$.
Then for all $z \in C_\epsilon( r(\epsilon) + 3 \epsilon)$, we have
\begin{equation}\label{eq wald type identity new}
\begin{aligned}
	&\mathbb{E}_{z} \left( \sigma^\epsilon(\vec{L}^\epsilon) \mathbbm{1}_{ \{Z^\epsilon(\sigma^\epsilon(\vec{L}^\epsilon)) \in C_{\epsilon}^k(L_k^\epsilon)\} } \right)
	=\sum_{l=1}^\infty \mathbb{E}_{z} \left (\sigma_{N^{\epsilon}}^{\epsilon,\delta}(\vec{L}^\epsilon)   \mathbbm{1}_{ \{N^{\epsilon}=l \} } \mathbbm{1}_{ \{Z^\epsilon(\sigma^\epsilon(\vec{L}^\epsilon)) \in C_{\epsilon}^k(L_k^\epsilon)\} } \right) \\
	&= \sum_{l=1}^\infty \mathbb{E}_z \left[  \sum _{i=0}^l  \left( \sigma_{i}^{\epsilon,\delta}(\vec{L}^\epsilon) - \tau_{i}^{\epsilon,\delta}(\vec{L}^\epsilon) \right) \mathbbm{1}_{ \{N^{\epsilon}=l \} } \mathbbm{1}_{ \{Z^\epsilon(\sigma^\epsilon(\vec{L}^\epsilon)) \in C_{\epsilon}^k(L_k^\epsilon)\} }  \right] \\
	& \ \ \ + \sum_{l=1}^\infty \mathbb{E}_z \left[  \sum _{i=1}^l \left( \tau_{i}^{\epsilon,\delta}(\vec{L}^\epsilon) - \sigma_{i-1}^{\epsilon,\delta}(\vec{L}^\epsilon) \right) \mathbbm{1}_{ \{N^{\epsilon}=l \} }  \mathbbm{1}_{ \{Z^\epsilon(\sigma^\epsilon(\vec{L}^\epsilon)) \in C_{\epsilon}^k(L_k^\epsilon)\} } \right]  \\
	&= \sum_{i=1}^\infty \sum _{l=i}^\infty \mathbb{E}_z \left[    \left( \sigma_{i}^{\epsilon,\delta}(\vec{L}^\epsilon) - \tau_{i}^{\epsilon,\delta}(\vec{L}^\epsilon) \right) \mathbbm{1}_{ \{N^{\epsilon}=l \} } \mathbbm{1}_{ \{Z^\epsilon(\sigma^\epsilon(\vec{L}^\epsilon)) \in C_{\epsilon}^k(L_k^\epsilon)\} } \right] \\
	& \ \ \ + \sum_{i=1}^\infty \sum _{l=i}^\infty \mathbb{E}_z \left[   \left( \tau_{i}^{\epsilon,\delta}(\vec{L}^\epsilon) - \sigma_{i-1}^{\epsilon,\delta}(\vec{L}^\epsilon) \right) \mathbbm{1}_{ \{N^{\epsilon}=l \} }  \mathbbm{1}_{ \{Z^\epsilon(\sigma^\epsilon(\vec{L}^\epsilon)) \in C_{\epsilon}^k(L_k^\epsilon)\} }  \right]  \\
	&=  \sum_{i=1}^\infty \mathbb{E}_z \left[    \left( \sigma_{i}^{\epsilon,\delta}(\vec{L}^\epsilon) - \tau_{i}^{\epsilon,\delta}(\vec{L}^\epsilon) \right) \mathbbm{1}_{ \{N^{\epsilon} \geq i \} } \mathbbm{1}_{ \{Z^\epsilon(\sigma^\epsilon(\vec{L}^\epsilon)) \in C_{\epsilon}^k(L_k^\epsilon)\} } \right] \\
	& \ \ \ + \sum_{i=1}^\infty \mathbb{E}_z \left[   \left( \tau_{i}^{\epsilon,\delta}(\vec{L}^\epsilon) - \sigma_{i-1}^{\epsilon,\delta}(\vec{L}^\epsilon) \right) \mathbbm{1}_{ \{N^{\epsilon} = i \} }  \mathbbm{1}_{ \{Z^\epsilon(\sigma^\epsilon(\vec{L}^\epsilon)) \in C_{\epsilon}^k(L_k^\epsilon)\} }  \right] \\
	& \ \ \ \ \ + \sum_{i=1}^\infty \mathbb{E}_z \left[   \left( \tau_{i}^{\epsilon,\delta}(\vec{L}^\epsilon) - \sigma_{i-1}^{\epsilon,\delta}(\vec{L}^\epsilon) \right) \mathbbm{1}_{ \{N^{\epsilon} \geq i+1 \} }  \mathbbm{1}_{ \{Z^\epsilon(\sigma^\epsilon(\vec{L}^\epsilon)) \in C_{\epsilon}^k(L_k^\epsilon)\} }  \right] \\
	& \eqqcolon I_1^\epsilon + I_2^\epsilon + I_3^\epsilon.
\end{aligned}
\end{equation}
Notice that $\{N^{\epsilon} \geq i+1 \} = \{N^{\epsilon} \leq i\}^c \in \mathcal{F}_{\sigma_{i}^{\epsilon,\delta}(\vec{L}^\epsilon)}$, by the strong Markov property, 
\begin{equation}\label{eq wald type identity new 2}
\begin{aligned}
	&I_3^\epsilon = \sum_{i=1}^\infty \mathbb{E}_z \left[   \left( \tau_{i}^{\epsilon,\delta}(\vec{L}^\epsilon) - \sigma_{i-1}^{\epsilon,\delta}(\vec{L}^\epsilon) \right) \mathbbm{1}_{ \{N^{\epsilon} \geq i+1 \} }  \mathbbm{1}_{ \{Z^\epsilon(\sigma^\epsilon(\vec{L}^\epsilon)) \in C_{\epsilon}^k(L_k^\epsilon)\} }  \right] \\
	& =  \sum_{i=1}^\infty \mathbb{E}_z \left[   \left( \tau_{i}^{\epsilon,\delta}(\vec{L}^\epsilon) - \sigma_{i-1}^{\epsilon,\delta}(\vec{L}^\epsilon) \right)  \mathbbm{1}_{ \{N^{\epsilon} \geq i+1 \} } 
	\mathbb{P}_{Z^\epsilon(\sigma_{i}^{\epsilon,\delta}(\vec{L}^\epsilon))} \left( Z^\epsilon(\sigma^\epsilon(\vec{L}^\epsilon)) \in C_{\epsilon}^k(L_k^\epsilon) \right) \right]  \\
	&\sim \sum_{i=1}^\infty \mathbb{E}_z \left[   \left( \tau_{i}^{\epsilon,\delta}(\vec{L}^\epsilon) - \sigma_{i-1}^{\epsilon,\delta}(\vec{L}^\epsilon) \right)  \mathbbm{1}_{ \{N^{\epsilon} \geq i+1 \} }  \right] p_{k}(\vec{L})
\end{aligned}
\end{equation}
since, for all $i < N^{\epsilon}$, $ Z^\epsilon(\sigma_{i}^{\epsilon,\delta}(\vec{L}^\epsilon)) \in C_{\epsilon}(r(\epsilon)+3\epsilon)$ and thus, by Lemma \ref{exit place est new}, we have
\begin{equation*}
	\mathbb{P}_{Z^\epsilon(\sigma_{i}^{\epsilon,\delta}(\vec{L}^\epsilon))} \left( Z^\epsilon(\sigma^\epsilon(\vec{L}^\epsilon)) \in C_{\epsilon}^k(L_k^\epsilon) \right)
	\sim p_{k}(\vec{L}), \ \ \ \mathbb{P} \ a.s.
\end{equation*}
Next, by the Wald's identity, Lemma \ref{exit place est} and Remark \ref{exit 1 cy}, we have
\begin{equation}\label{eq wald type identity new 3}
\begin{aligned}
	&\sum_{i=1}^\infty \mathbb{E}_z \left[   \left( \tau_{i}^{\epsilon,\delta}(\vec{L}^\epsilon) - \sigma_{i-1}^{\epsilon,\delta}(\vec{L}^\epsilon) \right)  \mathbbm{1}_{ \{N^{\epsilon} \geq i+1 \} }  \right] \\
	&= \sum_{i=1}^\infty \mathbb{E}_z \left[   \left( \tau_{i}^{\epsilon,\delta}(\vec{L}^\epsilon) - \sigma_{i-1}^{\epsilon,\delta}(\vec{L}^\epsilon) \right)  \mathbbm{1}_{ \{N^{\epsilon} \geq i \} }  \right]
	- \sum_{i=1}^\infty \mathbb{E}_z \left[   \left( \tau_{i}^{\epsilon,\delta}(\vec{L}^\epsilon) - \sigma_{i-1}^{\epsilon,\delta}(\vec{L}^\epsilon) \right)  \mathbbm{1}_{ \{N^{\epsilon} = i \} }  \right] \\
	& \sim \mathbb{E}_{z'} \sigma^{\epsilon}(\delta) \mathbb{E}_z N^{\epsilon}
	\sim \alpha(\epsilon) \frac{\sum_{l:I_l \sim O} \lambda_l^{d-1}}{ \sum_{k:I_k \sim O} \lambda_k^{d-1}/L_k},
\end{aligned}
\end{equation}
uniformly in $z,z' \in C_\epsilon(r(\epsilon)+3\epsilon)$.
Here we use that the second term is of lower order. More precisely, note that
\begin{equation}
    \{N^\epsilon > i-1\} = \{N^\epsilon \leq i-1\}^c \in \mathcal{F}_{\sigma_{i-1}^{\epsilon,\delta}(\vec{L}^\epsilon)} \subset \mathcal{F}_{\tau_{i}^{\epsilon,\delta}(\vec{L}^\epsilon)},
\end{equation}
and thus, for the second term,
\begin{equation*}
\begin{aligned}
	&\sum_{i=1}^\infty \mathbb{E}_z \left[   \left( \tau_{i}^{\epsilon,\delta}(\vec{L}^\epsilon) - \sigma_{i-1}^{\epsilon,\delta}(\vec{L}^\epsilon) \right)  \mathbbm{1}_{ \{N^{\epsilon} = i \} }  \right] \\
    & = \sum_{i=1}^\infty \mathbb{E}_z \left[   \left( \tau_{i}^{\epsilon,\delta}(\vec{L}^\epsilon) - \sigma_{i-1}^{\epsilon,\delta}(\vec{L}^\epsilon) \right)  \mathbbm{1}_{ \{N^{\epsilon} = i \} }  \mathbbm{1}_{ \{N^{\epsilon} > i-1 \} } \right] \\
    & = \sum_{i=1}^\infty \mathbb{E}_z \left[   \left( \tau_{i}^{\epsilon,\delta}(\vec{L}^\epsilon) - \sigma_{i-1}^{\epsilon,\delta}(\vec{L}^\epsilon) \right) \mathbbm{1}_{ \{N^{\epsilon} > i-1 \} } \mathbb{E}_{Z^\epsilon(\tau_i^{\epsilon,\delta}(\vec{L}^\epsilon))} \mathbbm{1}_{ \{N^{\epsilon} = 0 \} }   \right] \\
    &\lesssim \delta \sum_{i=1}^\infty \mathbb{E}_z \left[   \left( \tau_{i}^{\epsilon,\delta}(\vec{L}^\epsilon) - \sigma_{i-1}^{\epsilon,\delta}(\vec{L}^\epsilon) \right) \mathbbm{1}_{ \{N^{\epsilon} > i-1 \} } \right] \\
    & \lesssim \delta \sup_{z' \in C_\epsilon(r(\epsilon)+3\epsilon)} \mathbb{E}_{z'} \sigma^{\epsilon}(\delta) \sum_{i=1}^\infty \mathbb{P}_z \left( N^{\epsilon} \geq i \right) \\
	&= \delta \sup_{z' \in C_\epsilon(r(\epsilon)+3\epsilon)} \mathbb{E}_{z'} \sigma^{\epsilon}(\delta) \mathbb{E}_z N^\epsilon
	\ll_{\epsilon,\delta} \alpha(\epsilon) \frac{\sum_{l:I_l \sim O} \lambda_l^{d-1}}{ \sum_{k:I_k \sim O} \lambda_k^{d-1}/L_k}.
\end{aligned}
\end{equation*}
A similar argument shows that $I_2^\epsilon$ is of lower order.
The term $I_1^\epsilon$ involves only a one-dimensional Brownian motion, so a similar decomposition shows that it is also of lower order.

Finally, combining \eqref{eq wald type identity new}, \eqref{eq wald type identity new 2} and \eqref{eq wald type identity new 3}, we have
\begin{equation}
    \mathbb{E}_{z} \left( \sigma^\epsilon(\vec{L}^\epsilon) \mathbbm{1}_{ \{Z^\epsilon(\sigma^\epsilon(\vec{L}^\epsilon)) \in C_{\epsilon}^k(L_k^\epsilon)\} } \right)
    \sim_{\epsilon,\delta} \alpha(\epsilon) \frac{\sum_{l:I_l \sim O} \lambda_l^{d-1}}{ \sum_{k:I_k \sim O} \lambda_k^{d-1}/L_k} p_k(\vec{L}).
\end{equation}
Dividing both sides by $\mathbb{P}_z \left(Z^\epsilon(\sigma^\epsilon(\vec{L}^\epsilon)) \in C_{\epsilon}^k(L_k^\epsilon) \right)$ and applying Lemma \ref{exit place est new} again, we have the desired result.
\end{proof}

\section{Asymptotic exponential law for the exit time}\label{sec exit exp}
Since we continue to focus on the local behavior near a vertex, we omit the subscript $j$ throughout this section for clarity.

\begin{theorem}\label{large exit time est}
For each $T>0$,
	\begin{equation}\label{exp est 1}
		 \lim_{\epsilon \to 0} \sup_{z \in C_\epsilon( r(\epsilon)+3\epsilon)  } \sup_{t \in [0,T]} \abs{ \mathbb{P}_z \left( \sigma^{\epsilon}(\vec{L}^\epsilon) \geq t \alpha(\epsilon) \frac{\sum_{l:I_l \sim O} \lambda_l^{d-1}}{\sum_{k:I_k \sim O} \lambda_k^{d-1} / L_k} \right) - e^{-t}} = 0.
	\end{equation}
	Moreover, 
	\begin{equation}\label{exp est 2}
		 \lim_{\epsilon \to 0} \sup_{z \in C_\epsilon( r(\epsilon)+3\epsilon)  } \sup_{t \in [0,T]}  \abs{ \mathbb{P}_z \left( \theta_{k}^\epsilon(\vec{L}^\epsilon) \geq t \alpha(\epsilon) \frac{\sum_{l:I_l \sim O} \lambda_l^{d-1}}{\sum_{k:I_k \sim O} \lambda_k^{d-1} / L_k} \right) - e^{-t}} = 0,
	\end{equation}
    where 
    \begin{equation}
\alpha(\epsilon) = 
 \frac{r(\epsilon)^d V_{d}}{\sum_{k:I_k \sim O} \lambda_k^{d-1} \epsilon^{d-1} V_{d-1}}.
\end{equation}
	In particular, the exit time $\sigma^{\epsilon}(\vec{L}^\epsilon)$ and the exit edge $Z^\epsilon(\sigma^{\epsilon}(\vec{L}^\epsilon)) \in C_{\epsilon}^k(L_k^\epsilon)$ are asymptotically independent.
\end{theorem}

The proof follows the same approach as that of Section 7 in \cite{Hsu25}, which relies on Lemma 6.8 in \cite{SMP}.
We prove \eqref{exp est 2} by establishing Theorem \ref{thm abs} and Lemmas \ref{lem trans prob}, \ref{lem fir moment}, and \ref{lem sec moment} below. The other estimate follows from a similar argument.
To do so, recall that the sequence of stopping times
\begin{equation*}
\begin{aligned}
\sigma_{n}^{\epsilon,\delta}(\vec{L}^\epsilon) &\coloneqq \inf  \{t \geq  \tau_{n}^{\epsilon,\delta}(\vec{L}^\epsilon) : Z^\epsilon(t) \in  C_{\epsilon}(\vec{L}^\epsilon) \cup C_{\epsilon}( r(\epsilon)+ 3\epsilon ) \}, \\
\tau_{n}^{\epsilon,\delta}(\vec{L}^\epsilon) &\coloneqq \inf  \{t \geq  \sigma_{n-1}^{\epsilon,\delta}(\vec{L}^\epsilon) : Z^\epsilon(t) \in C_{\epsilon}(\delta) \},
\end{aligned}
\end{equation*}
with $\tau_{0}^{\epsilon,\delta}(\vec{L}^\epsilon) \coloneqq 0$.

Let $\mathcal{S}_{\epsilon} = C_\epsilon(r(\epsilon)+3\epsilon)$.
Define the discrete time Markov chain $ \hat{X}_n^{\epsilon,\delta,\vec{L}^\epsilon} \coloneqq Z^\epsilon( \sigma_{n}^{\epsilon,\delta}(\vec{L}^\epsilon) \wedge \sigma^\epsilon(\vec{L}^\epsilon) )$ and the filtration $ \mathcal{F}_n^{\epsilon,\delta,\vec{L}^\epsilon} \coloneqq \left\{ \sigma \left( \hat{X}_m^{\epsilon,\delta,\vec{L}^\epsilon} \right): m \leq n \right\}$.
Define $\xi_{k,n}^{\epsilon,\delta,\vec{L}^\epsilon}$ as follow,
\begin{equation*}
\begin{aligned}
  \xi_{k,0}^{\epsilon,\delta,\vec{L}^\epsilon} &= 0 , \\
 \xi_{k,n}^{\epsilon,\delta,\vec{L}^\epsilon} 
	&= 
    \left( \mathbbm{1}_{ \{ \hat{X}_{n-1}^{\epsilon,\delta,\vec{L}^\epsilon} \in C_\epsilon(\vec{L}^\epsilon) \}}e(\epsilon,\delta) + \mathbbm{1}_{ \{ \hat{X}_{n-1}^{\epsilon,\delta,\vec{L}^\epsilon} \in C_\epsilon( r(\epsilon)+3\epsilon)  \}} \left( \sigma_{n}^{\epsilon,\delta}(\vec{L}^\epsilon) - \sigma_{n-1}^{\epsilon,\delta}(\vec{L}^\epsilon) \right)  \right)  \\
	& \ \ \ \ \ \ \ \ \ \ \times  \frac{\mathbbm{1}_{\{ Z^\epsilon(\sigma^\epsilon(\vec{L}^\epsilon)) \in C_{\epsilon}^k(L_k^\epsilon)\}}}{\mathbb{P}_{\hat{X}_{0}^{\epsilon,\delta,\vec{L}^\epsilon}}(Z^\epsilon(\sigma^\epsilon(\vec{L}^\epsilon)) \in C_{\epsilon}^k(L_k^\epsilon))}, 
\end{aligned}
\end{equation*}
with 
\begin{equation}\label{value e}
e(\epsilon,\delta) = 
\frac{r(\epsilon)^d V_d}{\sum_{k:I_k \sim O} \lambda_k^{d-1} \epsilon^{d-1} V_{d-1}} \delta = \alpha(\epsilon)\delta .
\end{equation}
Let 
\begin{equation*}
\zeta_n^{\epsilon,\delta,\vec{L}^\epsilon} = 
\begin{dcases}
1, \text{ if } \hat{X}_{n}^{\epsilon,\delta,\vec{L}^\epsilon} \in C_\epsilon(r(\epsilon)+3\epsilon), \\
0, \text{ otherwise}.
\end{dcases}
\end{equation*}
Then it is clear that for $z \in \mathcal{S}_{\epsilon}$,
\begin{equation*}
	\sum_{n=0}^\infty \xi_{k,n+1}^{\epsilon,\delta,\vec{L}^\epsilon} \zeta_n^{\epsilon,\delta,\vec{L}^\epsilon} = \theta_{k}^\epsilon(\vec{L}^\epsilon), \ \ \ \mathbb{P}_z-a.s.
\end{equation*}

The following theorem closely parallels Lemma 7.1 of \cite{Hsu25} (see also Lemma 6.8 of \cite{SMP}), the only difference being the additional supremum over $t$. 
The proof follows the same lines as that of the original result.
The additional dependence on the parameter $t$ can be handled using an argument similar to that in Appendix D of \cite{Hsu25}.

\begin{theorem}\label{thm abs}
	Suppose that there is a non-negative function $e(\epsilon,\delta)$ such that
	\begin{equation}\label{fir moment}
	\lim_{\delta \to 0} \lim_{\epsilon \to 0}  \sup_{z \in \mathcal{S}_{\epsilon}} \sup_{ n \geq 0} \abs{ \frac{\mathbb{E}_z \left( \xi_{k,n+1}^{\epsilon,\delta,\vec{L}^\epsilon} \middle| \mathcal{F}_n^{\epsilon,\delta,\vec{L}^\epsilon} \right)}{e(\epsilon,\delta)} - 1} = 0,
	\end{equation}
	and 
	\begin{equation}\label{sec moment}
		\lim_{\delta \to 0} \lim_{\epsilon \to 0}  \sup_{z \in \mathcal{S}_{\epsilon}} \sup_{ n \geq 0} \mathbb{E}_z \left( \left( \xi_{k,n+1}^{\epsilon,\delta,\vec{L}^\epsilon} \right)^2 \middle| \mathcal{F}_{n}^{\epsilon,\delta,\vec{L}^\epsilon} \right) (e(\epsilon,\delta))^{-2} < \infty.
	\end{equation}
	Moreover, suppose that there is a function $p(\epsilon,\delta,\vec{L}^\epsilon) > 0$ such that $\lim_{\delta \to 0} \lim_{\epsilon \to 0} p(\epsilon,\delta,\vec{L}^\epsilon) = 0$ and
	\begin{equation}\label{trans prob}
		\lim_{\delta \to 0} \lim_{\epsilon \to 0}  \sup_{z \in \mathcal{S}_{\epsilon}} \sup_{ n \geq 0} \abs{ \frac{\mathbb{P}_z \left( \zeta_{n+1}^{\epsilon,\delta,\vec{L}^\epsilon}=0 \middle| \zeta_n^{\epsilon,\delta,\vec{L}^\epsilon} = 1 \right)}{p(\epsilon,\delta,\vec{L}^\epsilon)} - 1 }= 0.
	\end{equation}
	
	Let $K^{\epsilon,\delta,\vec{L}^\epsilon} = \min\{ n: \zeta_n^{\epsilon,\delta,\vec{L}^\epsilon} =0 \}$. Then for each $T \geq 0$,
	\begin{equation}\label{exp conv 1}
		\lim_{\delta \to 0} \lim_{\epsilon \to 0} \sup_{z \in \mathcal{S}_{\epsilon}} \sup_{t \in [0,T]} \abs{ \mathbb{P}_z (p(\epsilon,\delta,\vec{L}^\epsilon)K^{\epsilon,\delta,\vec{L}^\epsilon} \geq t) - e^{-t} } =0,
	\end{equation}
	and, moreover,
	\begin{equation}\label{exp conv 2}
		\lim_{\delta \to 0} \lim_{\epsilon \to 0} \sup_{z \in \mathcal{S}_{\epsilon}} \sup_{t \in [0,T]} \abs{ \mathbb{P}_z \left( \left( e(\epsilon,\delta) \right)^{-1} p(\epsilon,\delta,\vec{L}^\epsilon) \sum_{n=0}^\infty \xi_{k,n+1}^{\epsilon,\delta,\vec{L}^\epsilon} \zeta_n^{\epsilon,\delta,\vec{L}^\epsilon}  \geq t \right) - e^{-t}} = 0.
	\end{equation}
\end{theorem}

We will verify equations \eqref{fir moment}, \eqref{sec moment}, and \eqref{trans prob} with $e(\epsilon,\delta')$ defined by \eqref{value e}, and
\begin{equation}
p(\epsilon,\delta,\vec{L}^\epsilon) =  \frac{\sum_{k:I_k \sim O} \lambda_k^{d-1}/L_k}{\sum_{l:I_l \sim O} \lambda_l^{d-1}} \delta.
\end{equation}
Once this is established, note that
\begin{equation}
	\frac{e(\epsilon,\delta)}{p(\epsilon,\delta,\vec{L}^\epsilon)}
	=\alpha(\epsilon) \delta \left( \sum_{k:I_k \sim O} \frac{\lambda_k^{d-1}/L_k}{\sum_{l:I_l \sim O} \lambda_l^{d-1}} \delta \right)^{-1}
	=\alpha(\epsilon) \frac{\sum_{l:I_l \sim O} \lambda_l^{d-1}}{\sum_{k:I_k \sim O} \lambda_k^{d-1}/L_k},
\end{equation}
and thus \eqref{exp est 2} follows directly from Theorem \ref{thm abs}.

\begin{lemma}\label{lem trans prob}
We have
    \begin{equation*}
        \lim_{\delta \to 0} \lim_{\epsilon \to 0}  \sup_{z \in \mathcal{S}_{\epsilon}} \sup_{ n \geq 0} \abs{ \frac{\mathbb{P}_z \left( \zeta_{n+1}^{\epsilon,\delta,\vec{L}^\epsilon}=0 \middle| \zeta_n^{\epsilon,\delta,\vec{L}^\epsilon} = 1 \right)}{p(\epsilon,\delta,\vec{L}^\epsilon)} - 1 }= 0.
    \end{equation*}
\end{lemma}
\begin{proof}
By the strong Markov property, 
\begin{equation*}
	\mathbb{P}_z \left( \zeta_{n+1}^{\epsilon,\delta,\vec{L}^\epsilon}=0, \zeta_n^{\epsilon,\delta,\vec{L}^\epsilon} = 1 \right)
	= \mathbb{E}_z \left[ \mathbbm{1}_{ \{ \zeta_n^{\epsilon,\delta,\vec{L}^\epsilon} = 1 \} } \mathbb{P}_{Z^\epsilon(\sigma_n^{\epsilon,\delta}(\vec{L}^\epsilon))}  \left( Z^\epsilon(\sigma_1^{\epsilon}(\vec{L}^\epsilon)) \in C_\epsilon(\vec{L}^\epsilon)    \right) \right].
\end{equation*}
Under the event $\{ \zeta_n^{\epsilon,\delta,\vec{L}^\epsilon} = 1 \}$, $ Z^\epsilon(\sigma_n^{\epsilon,\delta}(\vec{L}^\epsilon)) \in C_\epsilon(r(\epsilon)+3\epsilon)$ and thus due to Remark \ref{exit 1 cy}, we have
\begin{equation*}
\mathbb{P}_{Z^\epsilon(\sigma_n^{\epsilon,\delta}(\vec{L}^\epsilon))}  \left( Z^\epsilon(\sigma_1^{\epsilon}(\vec{L}^\epsilon)) \in C_\epsilon(\vec{L}^\epsilon)    \right)
\sim_{\epsilon,\delta} p(\epsilon,\delta,\vec{L}^\epsilon).
\end{equation*}
	Dividing both sides by $\mathbb{P}_z \left(  \zeta_n^{\epsilon,\delta,\vec{L}^\epsilon} = 1 \right)$, we then have the result.
\end{proof}

\begin{lemma}\label{lem fir moment}
    We have
    \begin{equation*}
        \lim_{\delta \to 0} \lim_{\epsilon \to 0}  \sup_{z \in \mathcal{S}_{\epsilon}} \sup_{ n \geq 0} \abs{ \frac{\mathbb{E}_z \left( \xi_{k,n+1}^{\epsilon,\delta,\vec{L}^\epsilon} \middle| \mathcal{F}_n^{\epsilon,\delta,\vec{L}^\epsilon} \right)}{e(\epsilon,\delta)} - 1} = 0.
    \end{equation*}
\end{lemma}
\begin{proof}
By the strong Markov property,
\begin{equation*}
	\mathbb{E}_z \left( \xi_{k,n+1}^{\epsilon,\delta,\vec{L}^\epsilon} \middle| \mathcal{F}_n^{\epsilon,\delta,\vec{L}^\epsilon} \right)
	= \mathbb{E}_{\hat{X}_{n}^{\epsilon, \delta,\vec{L}^\epsilon}}  \xi_{k,1}^{\epsilon,\delta,\vec{L}^\epsilon}.
\end{equation*}
Thanks to the definition of $\xi_{k,n}^{\epsilon,\delta,\vec{L}^\epsilon}$, it suffices to consider $\hat{X}_{n}^{\epsilon, \delta',\vec{L}^\epsilon} \in C_\epsilon(r(\epsilon)+3\epsilon) $.
For $z' \in C_\epsilon(r(\epsilon)+3\epsilon)$, by the strong Markov property, we have
\begin{equation*}
\begin{aligned}
	& \mathbb{E}_{z'}  \xi_{k,1}^{\epsilon,\delta,\vec{L}^\epsilon} \\
	&= \mathbb{E}_{z'} \left[ \left( \sigma_{1}^{\epsilon,\delta}(\vec{L}^\epsilon)-  \tau_1^{\epsilon,\delta}(\vec{L}^\epsilon) \right) \mathbbm{1}_{\{ Z^\epsilon(\sigma^\epsilon(\vec{L}^\epsilon)) \in C_{\epsilon}^k(L_k^\epsilon)\}} \right] \left( \mathbb{P}_{z'}(Z^\epsilon(\sigma^\epsilon(\vec{L}^\epsilon)) \in C_{\epsilon}^k(L_k^\epsilon)) \right)^{-1} \\
	 & \ \ \ + \mathbb{E}_{z'} \left[ \left( \tau_{1}^{\epsilon,\delta}(\vec{L}^\epsilon)-  \sigma_0^{\epsilon,\delta}(\vec{L}^\epsilon) \right) \mathbbm{1}_{\{ Z^\epsilon(\sigma^\epsilon(\vec{L}^\epsilon)) \in C_{\epsilon}^k(L_k^\epsilon)\}} \right]
	\left( \mathbb{P}_{z'}(Z^\epsilon(\sigma^\epsilon(\vec{L}^\epsilon)) \in C_{\epsilon}^k(L_k^\epsilon)) \right)^{-1}.
\end{aligned}
\end{equation*}
For the first term, thanks to the property of the one dimensional Brownian motion and Lemma \ref{exit place est new},
\begin{equation*}
\begin{aligned}
	& \sup_{z' \in C_\epsilon(r(\epsilon)+3\epsilon)} \mathbb{E}_{z'} \left[ \left( \sigma_{1}^{\epsilon,\delta}(\vec{L}^\epsilon)-  \tau_1^{\epsilon,\delta}(\vec{L}^\epsilon) \right) \mathbbm{1}_{\{ Z^\epsilon(\sigma^\epsilon(\vec{L}^\epsilon)) \in C_{\epsilon}^k(L_k^\epsilon)\}} \right] \\
    & \ \ \ \ \ \ \ \ \ \ \ \ \ \ \ \ \ \ \ \ \ \ \ \  \times \left( \mathbb{P}_{z'}(Z^\epsilon(\sigma^\epsilon(\vec{L}^\epsilon)) \in C_{\epsilon}^k(L_k^\epsilon)) \right)^{-1} \\
	& \leq c \delta = o(e(\epsilon,\delta)).
\end{aligned}
\end{equation*}
For the second term, by the argument in Lemma \ref{small exit est 2 new}, we have
\begin{equation}\label{exp time est 1}
\begin{aligned}
	&\mathbb{E}_{z'} \left[ \left( \tau_{1}^{\epsilon,\delta}(\vec{L}^\epsilon)-  \sigma_0^{\epsilon,\delta}(\vec{L}^\epsilon) \right) \mathbbm{1}_{\{ Z^\epsilon(\sigma^\epsilon(\vec{L}^\epsilon)) \in C_{\epsilon}^k(L_k^\epsilon)\}} \right]
	\left( \mathbb{P}_{z'}(Z^\epsilon(\sigma^\epsilon(\vec{L}^\epsilon)) \in C_{\epsilon}^k(L_k^\epsilon)) \right)^{-1} \\
	&\sim_{\epsilon,\delta} \mathbb{E}_{z'} \left[ \tau_{1}^{\epsilon,\delta}(\vec{L}^\epsilon)-  \sigma_0^{\epsilon,\delta}(\vec{L}^\epsilon) \right] 
	\sim_{\epsilon,\delta} e(\epsilon, \delta)
\end{aligned}
\end{equation}
uniformly in $ z' \in C_\epsilon( r(\epsilon)+3\epsilon)$.
To be more precise,
\begin{equation}\label{exp time est 2}
\begin{aligned}
	&\mathbb{E}_{z'} \left[ \left( \tau_{1}^{\epsilon,\delta}(\vec{L}^\epsilon)-  \sigma_0^{\epsilon,\delta}(\vec{L}^\epsilon) \right) \mathbbm{1}_{\{ Z^\epsilon(\sigma^\epsilon(\vec{L}^\epsilon)) \in C_{\epsilon}^k(L_k^\epsilon)\}} \right] \\
	& = \mathbb{E}_{z'} \left[ \left( \tau_{1}^{\epsilon,\delta}(\vec{L}^\epsilon)-  \sigma_0^{\epsilon,\delta}(\vec{L}^\epsilon) \right) 
	\mathbb{P}_{Z^\epsilon(\tau_1^{\epsilon,\delta}(\vec{L}^\epsilon))} \left( Z^\epsilon(\sigma^\epsilon(\vec{L}^\epsilon)) \in C_{\epsilon}^k(L_k^\epsilon) \right) \right] \\
	& \sim_{\epsilon,\delta} \mathbb{E}_{z'} \left( \tau_{1}^{\epsilon,\delta}(\vec{L}^\epsilon)-  \sigma_0^{\epsilon,\delta}(\vec{L}^\epsilon) \right) p_{k}(\vec{L})
\end{aligned}
\end{equation}
uniformly in $ z' \in C_\epsilon( r(\epsilon)+3\epsilon)$.
Here, we use the fact that Lemma \ref{exit place est new} can be extended to an initial point in $C_\epsilon(\delta)$.
In fact,
\begin{equation}
\begin{aligned}
	&\mathbb{P}_{z} \left( Z^\epsilon(\sigma^\epsilon(\vec{L}^\epsilon)) \in C_{\epsilon}^k(L_k^\epsilon) \right) \\
	&= \mathbb{P}_{z} \left( Z^\epsilon(\sigma^\epsilon(\vec{L}^\epsilon)) \in C_{\epsilon}^k(L_k^\epsilon), Z^\epsilon(\sigma_0^{\epsilon,\delta}(\vec{L}^\epsilon)) \in C_\epsilon(r(\epsilon)+3\epsilon) \right) \\
	& \ \ \ \ + \mathbb{P}_{z} \left( Z^\epsilon(\sigma^\epsilon(\vec{L}^\epsilon)) \in C_{\epsilon}^k(L_k^\epsilon), Z^\epsilon(\sigma_0^{\epsilon,\delta}(\vec{L}^\epsilon)) \in C_\epsilon(\vec{L}^\epsilon) \right) \\
	& = \mathbb{P}_{z'} \left( Z^\epsilon(\sigma^\epsilon(\vec{L}^\epsilon)) \in C_{\epsilon}^k(L_k^\epsilon) \right) + o_{\epsilon,\delta}(1)
	\sim_{\epsilon,\delta} p_{k}(\vec{L}),
\end{aligned}
\end{equation}
uniformly in $z \in C_\epsilon(\delta)$, $ z' \in C_\epsilon( r(\epsilon)+3\epsilon)$.
Dividing \eqref{exp time est 2} by $\mathbb{P}_{z'}(Z^\epsilon(\sigma^\epsilon(\vec{L}^\epsilon)) \in C_{\epsilon}^k(L_k^\epsilon))$ and using Lemma \ref{exit place est new} again, we obtain \eqref{exp time est 1}.
\end{proof}

\begin{lemma}\label{lem sec moment}
    We have
    \begin{equation*}
       \lim_{\delta \to 0} \lim_{\epsilon \to 0}  \sup_{z \in \mathcal{S}_{\epsilon}} \sup_{ n \geq 0} \mathbb{E}_z \left( \left( \xi_{k,n+1}^{\epsilon,\delta,\vec{L}^\epsilon} \right)^2 \middle| \mathcal{F}_{n}^{\epsilon,\delta,\vec{L}^\epsilon} \right) (e(\epsilon,\delta))^{-2} < \infty.
    \end{equation*}
\end{lemma}
\begin{proof}
Similarly, it suffices to prove that for $\epsilon,\delta$ small enough
\begin{equation*}
	\sup_{z \in C_\epsilon(r(\epsilon)+3\epsilon)  } \mathbb{E}_{z} \left( \xi_{k,1}^{\epsilon,\delta,\vec{L}^\epsilon} \right)^2
    \leq c (e(\epsilon,\delta))^2  .
\end{equation*}
With Lemma \ref{exit place est new} and \ref{small exit est 2}, this follows from the same argument as in Lemma 4.2 of \cite{Hsu25}.
\end{proof}

\begin{remark}
    Combining Lemmas \ref{exit place est new}, \ref{small exit est 2 new}, and Theorem \ref{large exit time est}, one can check that the assumptions of Theorem 4.4 in \cite{SMP} are fulfilled.
    Consequently, the process exhibits the metastable behavior described in that theorem.
    However, our setting is more explicit, and thus we aim to obtain a more detailed description of the corresponding metastable distribution. 
    This will be the main objective of the remaining sections.
\end{remark}

\section{Intermediate time scales}\label{sec dis}
Let $\abs{V'}$ be the numbers of distinct orders $\{r_j(\epsilon)\}_{j=1}^{\abs{V}}$ and $r_{(i)}(\epsilon)$ be the $i$-th smallest order radius for $i=1,\cdots,\abs{V'}$.
In other words, for every $j=1,\cdots,\abs{V}$, there exists $i \in \{1,\cdots,\abs{V'}\}$ such that
\begin{equation*}
	r_j(\epsilon) \asymp r_{(i)}(\epsilon)
\end{equation*}
and for every $i=1,\cdots,\abs{V'}$, there exists $j \in \{1,\cdots,\abs{V}\}$ such that
\begin{equation*}
	r_{(i)}(\epsilon) \asymp r_j(\epsilon).
\end{equation*}
Moreover,
\begin{equation*}
	r_{(1)}(\epsilon) \ll r_{(2)}(\epsilon) \ll \cdots \ll r_{(\abs{V'})}(\epsilon).
\end{equation*}
For all $i=1,\cdots,\abs{V'}$, define
\begin{equation}
	T_\epsilon^i \coloneqq \frac{r_{(i)}(\epsilon)^d}{\epsilon^{d-1}}.
\end{equation}

In this section, we study the asymptotic behavior of $Z^\epsilon$ at the time scale $t(\epsilon)$ such that $T_\epsilon^i \ll t(\epsilon) \ll T_\epsilon^{i+1}$ for some $i=1,\cdots, \abs{V'}-1$.

\subsection{Asymptotic discrete time Markov chain}\label{sec dis markov}

For each $i=1,\cdots, \abs{V'}-1$, let $\{X_n^i\}_{n \in \{0\} \cup \mathbb{Z}^+}$ be a discrete time Markov chain with the state space $\mathcal{S} =  \bigcup_{j=1}^{\abs{V}} O_j$ and the transition probability 
\begin{equation}
	p^i(O_j,O_j) =
	\begin{dcases}
	1, &\text { if } r_{j}(\epsilon) \gg r_{(i)}(\epsilon), \\
	0, &\text{ otherwise}.
	\end{dcases}
	\end{equation}
	For $j \neq j'$ and $r_{j}(\epsilon) \lesssim r_{(i)}(\epsilon)$
	\begin{equation}
	p^i(O_{j},O_{j'}) = 
\begin{dcases}	
	p_{j,k}(\vec{L}_j), &\text{ if } I_k \sim O_{j},O_{j'}, \\
	0, &\text{ otherwise},
\end{dcases}
\end{equation}
where $p_{j,k}(\vec{L}_j)$ is given by Lemma \ref{exit place est new} with $L_{j,k} = \abs{I_k}$.
For every initial state $O_j$, the chain is absorbed in $\mathcal{S}$ almost surely; let $\mu^i(O_j,\cdot)$ denote the law of the absorbing state.

Let $\sigma^{\epsilon,i}$ be the first hitting time to $\bigcup_{j'} C_{\epsilon,j'}(r_{j'}(\epsilon)+3\epsilon)$, where the union is taken over all $j'$ such that $r_{j'}(\epsilon) \gg r_{(i)}(\epsilon)$.

In what follows, we consider 
\begin{equation}\label{L ep}
    L_{j,k}^\epsilon = \abs{I_k}-(r_{j'}(\epsilon)+3\epsilon), \ \ \ I_k \sim O_j, O_{j'}, \ \ \ \text{ and } \ \ \ \vec{L}_j^\epsilon = \bigcup_{k:I_k \sim O_j} \{ (k,L_{j,k}^\epsilon) \}.
\end{equation}

\begin{lemma}\label{lem dis mar}
	For every $i=1,\cdots, \abs{V'}-1$ and every $j$ such that $r_j(\epsilon) \lesssim r_{(i)}(\epsilon)$, we have, for all $j'=1,\cdots\abs{V}$,
	\begin{equation}\label{dis exit pl}
		\lim_{\epsilon \to 0} \sup_{z \in C_{\epsilon,j}(r_j(\epsilon)+3\epsilon)} \abs{ \mathbb{P}_{z}\left( Z^\epsilon(\sigma^{\epsilon,i}) \in C_{\epsilon,j'}(r_{j'}(\epsilon)+3\epsilon) \right) - \mu^i(O_j,O_{j'})} = 0,
	\end{equation}
	and for all $t(\epsilon) > 0$ satisfying $t(\epsilon) \gg T_{\epsilon}^i $,
	\begin{equation}\label{dis exit time lower}
		\lim_{\epsilon \to 0} \sup_{z \in C_{\epsilon,j}(r_j(\epsilon)+3\epsilon)} \mathbb{P}_{z}( \sigma^{\epsilon,i} \geq t(\epsilon) ) = 0.
	\end{equation}
\end{lemma}
\begin{proof}
Let $\tau_n^\epsilon$ denote the successive hitting times at which $Z^\epsilon$ visits different sets $C_{\epsilon,j}(r_j (\epsilon) + 3\epsilon)$. 
Specifically, set
$\tau_{0}^\epsilon \coloneqq 0$, and for $n \geq 1$, if $Z^\epsilon(\tau_{n-1}^\epsilon) \in C_{\epsilon,j}( r_j (\epsilon) + 3\epsilon)$, we define
\begin{equation}
	\tau_{n}^\epsilon \coloneqq \inf \left\{ t \geq \tau_{n-1}^\epsilon: Z^\epsilon(t) \in \bigcup_{i \neq j} C_{\epsilon,i}( r_i (\epsilon) + 3\epsilon) \right\}.
\end{equation}

	Equation \eqref{dis exit pl} follows directly from Lemma \ref{exit place est new} by taking $\vec{L}_j^\epsilon$ as in \eqref{L ep} and applying the definition of the transition probabilities of $X_n^i$.
	
	For \eqref{dis exit time lower}, note that, by the connectedness of the graph, for all $j$ such that $r_j(\epsilon) \lesssim r_{(i)}(\epsilon)$ there exist constants $c_0>0$ and $N>0$, such that
	\begin{equation}
		\inf_{\epsilon \in (0,1)} \inf_{z \in C_{\epsilon,j}(r_{j}(\epsilon)+3\epsilon)} \mathbb{P}_{z}( \sigma^{\epsilon,i} < \tau_N^\epsilon ) \geq c_0.
	\end{equation}
	Consequently, for all $j$ such that $r_j(\epsilon) \lesssim r_{(i)}(\epsilon)$,
	\begin{equation}
	\begin{aligned}
		&\sup_{z \in  C_{\epsilon,j}(r_{j}(\epsilon)+3\epsilon)} \mathbb{P}_{z}(\sigma^{\epsilon,i} \geq t(\epsilon)) \\
		&\leq \sum_{m=1}^{M} \sup_{z \in  C_{\epsilon,j}(r_{j}(\epsilon)+3\epsilon)} \mathbb{P}_{z}(\sigma^{\epsilon,i} \geq t(\epsilon),  \sigma^{\epsilon,i} = \tau_m^\epsilon ) \\
		& \ \ \ + \sum_{m=M+1}^\infty \sup_{z \in  C_{\epsilon,j}(r_{j}(\epsilon)+3\epsilon)} \mathbb{P}_{z}(\sigma^{\epsilon,i} \geq t(\epsilon), \sigma^{\epsilon,i} = \tau_m^\epsilon ) \\
		& \eqqcolon I_1^\epsilon(M) + I_2^\epsilon(M).
	\end{aligned}
	\end{equation}
	For $I_2^\epsilon$, first note that for every $m \geq 0$, by the strong Markov property,
	\begin{equation}
		 \sup_{z \in  C_{\epsilon,j}(r_{j}(\epsilon)+3\epsilon)} \mathbb{P}_{z}(\sigma^{\epsilon,i} \geq t(\epsilon), \sigma^{\epsilon,i} = \tau_m^\epsilon)
         \leq \sup_{z \in  C_{\epsilon,j}(r_{j}(\epsilon)+3\epsilon)} \mathbb{P}_{z}(\sigma^{\epsilon,i} \geq \tau_m^\epsilon)
         \leq (1-c_0)^{ \lfloor  \frac{m}{N} \rfloor  }.
	\end{equation}
	For every $\eta >0$, there exists $M_\eta >0$ such that
    \begin{equation}\label{dis eq 1}
        I_2^\epsilon(M_\eta) \leq \sum_{m=M_\eta+1}^\infty (1-c_0)^{ \lfloor  \frac{m}{N} \rfloor  } < \eta.
    \end{equation}
	Next, for $I_1^\epsilon$, note that, by the strong Markov property and Lemma \ref{small exit est 2}
	\begin{equation}
    \begin{aligned}
		&\sup_{z \in C_{\epsilon,j}(r_{j}(\epsilon)+3\epsilon)} \mathbb{P}_{z}(\sigma^{\epsilon,i} = \tau_{m}^\epsilon,  \tau_{m}^\epsilon \geq t(\epsilon))
        \leq \sup_{z \in C_{\epsilon,j}(r_{j}(\epsilon)+3\epsilon)}  \frac{\mathbb{E}_z \tau_m^\epsilon \mathbbm{1}_{ \{ \sigma^{\epsilon,i} = \tau_m^\epsilon \} } }{t(\epsilon)} \\
		& = \sup_{z \in C_{\epsilon,j}(r_{j}(\epsilon)+3\epsilon)}  \frac{\mathbb{E}_z \tau_m^\epsilon \mathbbm{1}_{ \{ \sigma^{\epsilon,i} = \tau_m^\epsilon \} } \mathbbm{1}_{ \{ \sigma^{\epsilon,i} > \tau_{m-1}^\epsilon \} } }{t(\epsilon)} \\
        &\leq m  \frac{ \max_{j'} \sup_{z \in C_{\epsilon,j'}(r_{j'}(\epsilon)+3\epsilon)} \mathbb{E}_{z} \sigma_{j'}^\epsilon(\vec{L}_{j'}^\epsilon)}{t(\epsilon)}
		\leq c m  \frac{T_\epsilon^i}{t(\epsilon)},
    \end{aligned}
	\end{equation}
	where the maximum is taken over all $j'$ such that $r_{j'}(\epsilon) \lesssim r_{(i)}(\epsilon)$.
	As a result, we have
	\begin{equation}\label{dis eq 2}
	\lim_{\epsilon \to 0} I_1^\epsilon(M_\eta) \leq \lim_{\epsilon \to 0} \sum_{m=1}^{M_\eta}  c m  \frac{T_\epsilon^i}{t(\epsilon)}= 0.
	\end{equation}
    Combining \eqref{dis eq 1} and \eqref{dis eq 2}, we have
    \begin{equation}
        \lim_{\epsilon \to 0} \sup_{z \in  C_{\epsilon,j}(r_{j}(\epsilon)+3\epsilon)} \mathbb{P}_{z}(\sigma^{\epsilon,i} \geq t(\epsilon)) \leq \eta.
    \end{equation}
	Since $\eta >0$ is arbitrary, we conclude that \eqref{dis exit time lower} holds.
\end{proof}

\begin{lemma}
    For every $i=1,\cdots,\abs{V'}-1$, every $t(\epsilon) > 0$ with $t(\epsilon) \ll T_{\epsilon}^{i+1}$, every $j$ such that $r_j(\epsilon) \gg r_{(i)}(\epsilon)$, and all $\delta > 0$, we have
	\begin{equation}\label{dis exit time upper}
		\lim_{\epsilon \to 0} \sup_{z \in  C_{\epsilon,j}(r_{j}(\epsilon)+3\epsilon)} \mathbb{P}_{z}( \sigma_j^\epsilon(\delta) \leq t(\epsilon) ) = 0.
	\end{equation}
\end{lemma}
\begin{proof}
    For all $\eta >0$, let $t_\eta > 0$ such that $e^{-t_\eta} \geq 1-\frac{\eta}{2}$.
	For all $j$ such that $r_j(\epsilon) \gg r_{(i)}(\epsilon)$, there exist $\epsilon_{0} >0$ and $c >0$ such that for all $\epsilon \in (0,\epsilon_{0})$,
	\begin{equation}
		\alpha_j(\epsilon)  = \frac{r_j(\epsilon)^d V_d}{\sum_{k:I_k \sim O_j} \lambda_k^{d-1} \epsilon^{d-1} V_{d-1}} \geq c T_\epsilon^{i+1}.
	\end{equation}
	Moreover, there exists $\epsilon_{1,\eta} \in (0,\epsilon_{0})$ such that for all $\epsilon \in (0,\epsilon_{1,\eta})$,
	\begin{equation}
		\frac{t(\epsilon)}{\alpha_j(\epsilon)\delta} \leq \frac{t(\epsilon)}{cT_\epsilon^{i+1}\delta} < t_\eta.
	\end{equation}
	We have
	\begin{equation}
		\mathbb{P}_z (\sigma_j^\epsilon(\delta) \geq t(\epsilon))
		\geq \mathbb{P}_z (\sigma_j^\epsilon(\delta) \geq t_\eta \alpha_j(\epsilon)\delta).
	\end{equation}
	By \eqref{exp est 1}, there exists $\epsilon_\eta \in (0,\epsilon_{1,\eta})$ such that
	\begin{equation}
		\mathbb{P}_z (\sigma^\epsilon(\delta) \geq t_\eta \alpha_j(\epsilon)\delta) \geq e^{-t_\eta}-\frac{\eta}{2} \geq 1-\eta.
	\end{equation}
	We then conclude that
	\begin{equation}
		\mathbb{P}_z (\sigma^\epsilon(\delta) \leq t(\epsilon))
		\leq 1-\mathbb{P}_z (\sigma^\epsilon(\delta) \geq t_\eta \alpha_j(\epsilon)\delta)
		\leq \eta.
	\end{equation}
\end{proof}

\begin{theorem}\label{conv diff}
	For every $i=1,\cdots,\abs{V'}-1$, every $t(\epsilon) > 0$ such that $T_{\epsilon}^i \ll t(\epsilon) \ll T_{\epsilon}^{i+1}$ and any continuous function $F \in C(\Gamma)$, we have
	\begin{equation}
		\lim_{\epsilon \to 0} \sup_{z \in C_{\epsilon,j}(r_j(\epsilon)+3\epsilon)} \abs{ \mathbb{E}_{z} (F(\Pi(Z^\epsilon(t(\epsilon)))) - \sum_{j'=1}^{\abs{V}} F(O_{j'}) \mu^i(O_j,O_{j'}) } = 0.
	\end{equation}
\end{theorem}
\begin{proof}
First, consider $z \in C_{\epsilon,j}(r_j(\epsilon)+3\epsilon)$ such that $r_j(\epsilon) \lesssim r_{(i)}(\epsilon)$.
Note that
\begin{equation}
\begin{aligned}
	& \mathbb{E}_{z} F(\Pi(Z^\epsilon(t(\epsilon))) \\
	&= \mathbb{E}_{z} \left( F(\Pi(Z^\epsilon(t(\epsilon)))) \mathbbm{1}_{\{ \sigma^{\epsilon,i} \geq t(\epsilon) \}} \right)
	+ \mathbb{E}_{z} \left( F(\Pi(Z^\epsilon(t(\epsilon)))) \mathbbm{1}_{\{ \sigma^{\epsilon,i} < t(\epsilon) \}} \right) \\
	&\eqqcolon I_{1}^\epsilon + I_2^\epsilon.
\end{aligned}
\end{equation}
For $I_1^\epsilon$, thanks to \eqref{dis exit time lower}, we have
\begin{equation}
\begin{aligned}
	&\sup_{z \in C_{\epsilon,j}(r_j(\epsilon)+3\epsilon)} \abs{\mathbb{E}_{z} \left( F(\Pi(Z^\epsilon(t(\epsilon)))) \mathbbm{1}_{\{ \sigma^{\epsilon,i} \geq t(\epsilon) \}} \right)} \\
	&\leq \abs{F}_{C(\Gamma)} \sup_{z \in C_{\epsilon,j}(r_j(\epsilon)+3\epsilon)}  \mathbb{P}_{z}( \sigma^{\epsilon,i} \geq t(\epsilon) )
	\to 0, \text{ as } \epsilon \to 0.    
\end{aligned}
\end{equation}
For $I_2^\epsilon$, first define the stopping time
\begin{equation*}
	\tau^{\epsilon,i}(\delta) \coloneqq \inf \{ t \geq \sigma^{\epsilon,i}: Z^\epsilon(t) \in \bigcup_{j=1}^{\abs{V}} C_{\epsilon,j}(\delta) \}.
\end{equation*}
We obtain
\begin{equation}
\begin{aligned}
	&I_2^\epsilon = \mathbb{E}_{z} \left( F(\Pi(Z^\epsilon(t(\epsilon)))) \mathbbm{1}_{\{ \sigma^{\epsilon,i} < t(\epsilon) \}} \right) \\
	&= \mathbb{E}_{z} \left( [F(\Pi(Z^\epsilon(t(\epsilon))))- F(\Pi(Z^\epsilon(\sigma^{\epsilon,i}))) ] \mathbbm{1}_{\{ \sigma^{\epsilon,i} < t(\epsilon) \}} \right) \\
	& \ \ \ + \mathbb{E}_{z} \left( F(\Pi(Z^\epsilon(\sigma^{\epsilon,i}))) \mathbbm{1}_{\{ \sigma^{\epsilon,i} < t(\epsilon) \}} \right)\\
	& = \mathbb{E}_{z} \left( [F(\Pi(Z^\epsilon(t(\epsilon)))- F(\Pi(Z^\epsilon(\sigma^{\epsilon,i}))) ] \mathbbm{1}_{\{ \sigma^{\epsilon,i} < t(\epsilon) \}} \mathbbm{1}_{\{ \tau^{\epsilon,i}(\delta) > t(\epsilon) \}}  \right) \\
	& \ \ \ + \mathbb{E}_{z} \left( [F(\Pi(Z^\epsilon(t(\epsilon))))- F(\Pi(Z^\epsilon(\sigma^{\epsilon,i}))) ] \mathbbm{1}_{\{ \sigma^{\epsilon,i} < t(\epsilon) \}} \mathbbm{1}_{\{ \tau^{\epsilon,i}(\delta) \leq t(\epsilon) \}}  \right) \\
	& \ \ \ \ \ + \mathbb{E}_{z} \left( F(\Pi(Z^\epsilon(\sigma^{\epsilon,i}))) \mathbbm{1}_{\{ \sigma^{\epsilon,i} < t(\epsilon) \}} \right) - \mathbb{E}_{z} \left( F(\Pi(Z^\epsilon(\sigma^{\epsilon,i})))  \right) \\
	& \ \ \ \ \ \ \ + \mathbb{E}_{z} \left( F(\Pi(Z^\epsilon(\sigma^{\epsilon,i})))  \right) \\
	& \ \ \ \ \ \ \ \ \   \eqqcolon J_{2,1}^\epsilon + J_{2,2}^\epsilon + J_{2,3}^\epsilon + J_{2,4}^\epsilon.
\end{aligned}
\end{equation}
For $J_{2,1}^\epsilon$, by the continuity of $F$, we have
\begin{equation}
\begin{aligned}
	&\sup_{z \in C_{\epsilon,j}(r_j(\epsilon)+3\epsilon)} \abs{ \mathbb{E}_{z} \left( [F(\Pi(Z^\epsilon(t(\epsilon))))- F(\Pi(Z^\epsilon(\sigma^{\epsilon,i}))) ] \mathbbm{1}_{\{ \sigma^{\epsilon,i} < t(\epsilon) \}} \mathbbm{1}_{\{ \tau^{\epsilon,i}(\delta) > t(\epsilon) \}}  \right)	} \\
	& \leq \sup_{z \in C_{\epsilon,j}(r_j(\epsilon)+3\epsilon)} \mathbb{E}_{z} \abs{ [F(\Pi(Z^\epsilon(t(\epsilon))))- F(\Pi(Z^\epsilon(\sigma^{\epsilon,i}))) ] \mathbbm{1}_{\{ \sigma^{\epsilon,i} < t(\epsilon) \}} \mathbbm{1}_{\{ \tau^{\epsilon,i}(\delta) > t(\epsilon) \}}  } \\
	&= o_{\epsilon,\delta}(1).
	\end{aligned}
\end{equation}
For $J_{2,2}^\epsilon$, due to \eqref{dis exit time upper}, we get
\begin{equation}
\begin{aligned}
	&\sup_{z \in C_{\epsilon,j}(r_j(\epsilon)+3\epsilon)} \abs{\mathbb{E}_{z} \left( [F(\Pi(Z^\epsilon(t(\epsilon))))- F(\Pi(Z^\epsilon(\sigma^{\epsilon,i}))) ] \mathbbm{1}_{\{ \sigma^{\epsilon,i} < t(\epsilon) \}} \mathbbm{1}_{\{ \tau^{\epsilon,i}(\delta) \leq t(\epsilon) \}}  \right)} \\
	& \leq 2\abs{F}_{C(\Gamma)}
	\sup_{z \in C_{\epsilon,j}(r_j(\epsilon)+3\epsilon)} \mathbb{E}_{z}  \left[ \mathbbm{1}_{\{ \sigma^{\epsilon,i} < t(\epsilon) \}} \mathbbm{1}_{\{ \tau^{\epsilon,i}(\delta) \leq t(\epsilon) \}} \right] \\
	& \leq 2\abs{F}_{C(\Gamma)}
	\sup_{z \in C_{\epsilon,j}(r_j(\epsilon)+3\epsilon)} \mathbb{E}_{z}  \bigg[ \mathbbm{1}_{\{ \sigma^{\epsilon,i} < t(\epsilon) \}} \sum_{j'}  \mathbbm{1}_{\{ Z^\epsilon(\sigma^{\epsilon,i}) \in C_{\epsilon,j'}(r_{j'}(\epsilon)+3\epsilon) \}} \\
    & \ \ \ \ \ \ \ \ \ \ \ \ \ \ \ \times \mathbb{P}_{Z^\epsilon(\sigma^{\epsilon,i})} ( \sigma_{j'}^{\epsilon}(\delta) \leq t(\epsilon) ) \bigg] \\
	& \leq 2\abs{F}_{C(\Gamma)} \max_{j'} \sup_{z \in C_{\epsilon,j'}(r_{j'}(\epsilon)+3\epsilon)} \mathbb{P}_{z} (\sigma_{j'}^\epsilon(\delta) \leq t(\epsilon)) \to 0 \text{ as } \epsilon \to 0,
\end{aligned}
\end{equation}
where the maximum is taken over all $j'$ such that $r_{j'}(\epsilon) \gg r_{(i)}(\epsilon)$.
For $J_{2,3}^\epsilon$, by \eqref{dis exit time lower}, we have
\begin{equation}
\begin{aligned}
	&\sup_{z \in C_{\epsilon,j}(r_j(\epsilon)+3\epsilon)} \abs{\mathbb{E}_{z} \left( F(\Pi(Z^\epsilon(\sigma^{\epsilon,i}))) \mathbbm{1}_{\{ \sigma_\epsilon^{i} < t(\epsilon) \}} \right) - \mathbb{E}_{z} \left( F(\Pi(Z^\epsilon(\sigma^{\epsilon,i})))  \right)}
	\\
	& \leq \sup_{z \in C_{\epsilon,j}(r_j(\epsilon)+3\epsilon)} \mathbb{E}_{z} \abs{F(\Pi(Z^\epsilon(\sigma^{\epsilon,i})) \mathbbm{1}_{\{ \sigma^{\epsilon,i} \geq t(\epsilon) \}} } \\
	& \leq \abs{F}_{C(\Gamma)} \sup_{z \in C_{\epsilon,j}(r_j(\epsilon)+3\epsilon)} \mathbb{P}_{z} (\sigma^{\epsilon,i} \geq t(\epsilon)) \to 0 \text{ as } \epsilon \to 0.
\end{aligned}
\end{equation}
Finally, combining all of the above, we have
\begin{equation}
\begin{aligned}
	&\sup_{z \in C_{\epsilon,j}(r_j(\epsilon)+3\epsilon)} \abs{ \mathbb{E}_{z} (F(\Pi(Z^\epsilon(t(\epsilon)))) - \sum_{j'=1}^{\abs{V}} F(O_{j'}) \mu^i(O_j,O_{j'}) } \\
	&\leq \sup_{z \in C_{\epsilon,j}(r_j(\epsilon)+3\epsilon)} \abs{ \mathbb{E}_{z} (F(\Pi(Z^\epsilon(\sigma^{\epsilon,i}))) - \sum_{j'=1}^{\abs{V}} F(O_{j'}) \mu^i(O_j,O_{j'}) } +o_{\epsilon,\delta}(1).
\end{aligned}
\end{equation}
By the continuity of $F$ and \eqref{dis exit pl}, we can conclude the result.

For $z \in C_{\epsilon,j}(r_j(\epsilon)+3\epsilon)$ such that $r_j(\epsilon) \gg r_{(i)}(\epsilon)$, the argument is similar.
\end{proof}

We aim to extend the result to all $z \in G_\epsilon$. 
To do so, first, if $x \in I_k$, which possess $O_{j},O_{j'}$ as its endpoints.
We define
\begin{equation}
	p(x,O_{j}) = \frac{d(x,O_{j'})}{\abs{I_k}},
\end{equation}
which represents the probability that a one-dimensional Brownian motion on $I_k$, starting at $x$, hits $O_{j}$ before reaching $O_{j'}$.
In addition, for $O_{j''} \neq O_{j},O_{j'}$, we define
\begin{equation}
    p(x,O_{j''}) = 0.
\end{equation}
For each $x \in \Gamma$, we introduce the measure $\mu^i(x,\cdot)$ on $\mathcal{S}$, defined as follows:
\begin{equation}
	\mu^i(x,O_{j'}) = \sum_{j=1}^{\abs{V}} p(x,O_j) \mu^i(O_j,O_{j'}).
\end{equation}

\begin{corollary}\label{coro conv diff}
    For every $i=1,\cdots,\abs{V'}-1$, any $t(\epsilon) > 0$ such that $T_\epsilon^i \ll t(\epsilon) \ll T_\epsilon^{i+1}$ and any continuous function $F \in C(\Gamma)$, we have
	\begin{equation}
		\lim_{\epsilon \to 0} \sup_{x \in \Gamma} \sup_{z: \Pi(z)=x} \abs{ \mathbb{E}_{z} (F(\Pi(Z^\epsilon(t(\epsilon)))) - \sum_{j'=1}^{\abs{V}} F(O_{j'}) \mu^i(x,O_{j'}) } = 0.
	\end{equation}
\end{corollary}
\begin{proof}
	For every $z \in G_\epsilon$, consider the first hitting time to $\bigcup_{j=1}^{\abs{V}} C_{\epsilon,j}(r_j(\epsilon)+3\epsilon)$.
	The results follows from the definition of $\mu^i(x,\cdot)$ and the property of one-dimensional Brownian motion.
\end{proof}

\subsection{Neumann problems}\label{sec pde}
In this section, we analyze the Cauchy linear problem associated with Laplacian $\Delta$ on the narrow tube $G_\epsilon$
\begin{equation}
\label{narrow PDE C}
\begin{aligned}
    \begin{dcases}
    \displaystyle
        \frac{\partial \rho_\epsilon }{\partial t} (t,z)
        = \Delta \rho_\epsilon (t,z),
        \   z \in G_\epsilon, \ t \geq 0 \\
        \frac{\partial \rho_\epsilon }{\partial \nu_\epsilon} (t,z) = 0 , \ z \in 	\partial G_\epsilon, \ t \geq 0, \ \ \ \
        \rho_\epsilon(0,z)=\varphi_\epsilon (z), \ z \in G_\epsilon, 
    \end{dcases}
\end{aligned}
\end{equation}
where $\varphi_\epsilon \in C(\bar{G_\epsilon})$.

It is well-known that the solution $\rho_\epsilon(t)$ to the equation \eqref{narrow PDE C} has a probabilistic representation in terms of solution to the diffusion process with reflection \eqref{diff eq}. 
Namely, it holds,
\begin{equation*}
\rho_\epsilon(t,z) = S_\epsilon(t) \varphi_\epsilon(z) = \mathbb{E}_z \varphi_\epsilon(Z^\epsilon(t)), \  \ \ t \geq 0, \ \ z \in G_\epsilon.
\end{equation*}
Our goal is to study the limiting behavior of $\rho_\epsilon(t)$.
We have already shown the convergence holds when $\varphi_\epsilon = F \circ \Pi$ for some $F \in C(\Gamma)$.
In this section, we prove the convergence holds for different initial conditions.

\begin{definition}
	We say a family of functions $\{ \varphi_\epsilon \}_{\epsilon \in (0,1)}$ with $\varphi_\epsilon:G_\epsilon \to \mathbb{R}$ is equicontinuous, if for every $\eta>0$, there exists $\delta_\eta >0$ such that for every $\delta \in (0,\delta_\eta)$ and all $z,z' \in G_\epsilon$ such that $\abs{z-z'} \leq \delta$,
	we have
	\begin{equation*}
		\abs{\varphi_\epsilon(z)-\varphi_\epsilon(z')} \leq \eta.
	\end{equation*}
\end{definition}

\begin{definition}
	We say a family of continuous functions $\{ \varphi_\epsilon \}_{\epsilon \in (0,1)}$ with $\varphi_\epsilon:G_\epsilon \to \mathbb{R}$ is equibounded if there exists $M>0$, such that
	\begin{equation*}
		\sup_{\epsilon \in (0,1)} \abs{\varphi_\epsilon}_{C(\bar{G}_\epsilon)} \leq M .
	\end{equation*}
\end{definition}

In particular, it is clear that $\varphi_\epsilon=\varphi|_{G_\epsilon}$ for some $\varphi \in C(\bar{G})$ is equicontinuous and equibounded, where $G \coloneqq G_1$.

\begin{corollary}\label{coro conv pde spec}
    Let $\{ \varphi_\epsilon \}_{\epsilon \in (0,1)}$ be equicontinuous and equibounded.
    For every $i=1,\cdots,\abs{V'}-1$ and any $t(\epsilon) > 0$ such that $T_\epsilon^i \ll t(\epsilon) \ll T_\epsilon^{i+1}$, we have
    \begin{equation}
        \lim_{\epsilon \to 0} \sup_{x \in \Gamma} \sup_{z: \Pi(z)=x} \abs{
        \mathbb{E}_{z} \varphi_\epsilon (Z^\epsilon(t(\epsilon))) 
        - \sum_{j'=1}^{\abs{V}} \varphi_\epsilon (O_{j'}) \mu^i(x,O_{j'})
        } =0.
    \end{equation}
\end{corollary}
\begin{proof}
	The proof is similar as the one in Theorem \ref{conv diff}. In this case, we use the equicontinuity and equiboundedness of $\varphi_\epsilon$.
\end{proof}

\section{First critical time scale}\label{sec conti}
Throughout this section, we consider the case that there exists only one $O_{j_1} \in V $ such that $r_{j_1}(\epsilon) = r_{(1)}(\epsilon) $ and for all $j \neq j_1$, $r_{(1)}(\epsilon) \ll r_j(\epsilon)$.
We will discuss the asymptotic behavior at time scale 
\begin{equation}
\begin{aligned}
	&t(\epsilon) = s T_\epsilon^1 = s  \frac{V_{d-1} \sum_{k:I_k \sim O_{j_1}} \lambda_k^{d-1}/L_k }{V_d}  \alpha_{j_1}(\epsilon) \frac{\sum_{l:I_l \sim O_{j_1}} \lambda_l^{d-1}}{\sum_{k:I_k \sim O_{j_1}} \lambda_k^{d-1}/L_k } \\
    &\eqqcolon s \kappa_{j_1}(\vec{L}) \alpha_{j_1}(\epsilon) \frac{\sum_{l:I_l \sim O_{j_1}} \lambda_l^{d-1}}{\sum_{k:I_k \sim O_{j_1}} \lambda_k^{d-1}/L_k },
\end{aligned}
\end{equation}
where
\begin{equation}
    T_\epsilon^1 = \frac{r_{j_1}(\epsilon)^d}{\epsilon^{d-1}}, \ \ \ \kappa_{j_1}(\vec{L}) = \frac{V_{d-1} \sum_{k:I_k \sim O_{j_1}} \lambda_k^{d-1}/L_k }{V_d}.
\end{equation}

\subsection{Asymptotic continuous time Markov chain}
Consider the following continuous time Markov chain $\{Y(t)\}_{t \geq 0}$ on the state space $\mathcal{S}= \bigcup_{j=1}^{\abs{V}} O_j $ defined on the probability space $(\mathbf{\Omega}, \mathbf{P})$:
first, let the random variables $\tau_n^{O_{j}} $ for $n \in \mathbb{N}$ be i.i.d. with exponential distribution with parameter 
\begin{equation}
  \kappa_{j_1}(\vec{L}) \mathbbm{1}_{\{j = j_1\}},
\end{equation}
and $\eta_n^{O_{j}} $ take value in $\mathcal{S} \setminus \{O_{j}\}$ are i.i.d. with 
\begin{equation}
\mathbf{P}(\eta_n^{O_{j}} = O_i) =
\begin{dcases}
    p_{j,k}(\vec{L}_j) = \frac{\lambda_k^{d-1}/L_k}{\sum_{l:I_l \sim O_{j}} \lambda_l^{d-1}/L_l }, & \text{ if } O_i \sim I_k \sim O_{j}, \\
    0, & \text{ otherwise}.
\end{dcases}
\end{equation}
Moreover, $\tau_n^{O_j}$ and $\eta_n^{O_j}$ are independent.
We define $\sigma_n$ with $n \geq 0$ and values in $[0,\infty)$ and $\xi_n$ with $n \geq 0$ and values in $\mathcal{S}$.
Let $\sigma_0 \coloneqq 0$ and $\xi_0 = \xi$.
We define
\begin{equation}
\begin{aligned}
\sigma_n = \sigma_{n-1} + \tau_n^{\xi^{n-1}}, \ \ \ \ \  \xi_n = \eta_n^{\xi_{n-1}}.
\end{aligned}
\end{equation}
Finally, 
\begin{equation}
	Y(t) = \xi_n , \ \ \ \ \text{ for } \sigma_n \leq t <\sigma_{n+1}.
\end{equation}
We denote the expectation associate with the process $\{Y(t)\}_{t \geq 0}$ starting from $O_j \in V$ by $\mathbf{E}_{O_j}$. 

Due to Lemma \ref{exit place est new} and Theorem \ref{large exit time est}, it is natural to expect that $Z^\epsilon(sT_\epsilon^1)$ behaves like $Y(s)$ when $\epsilon$ is small.
The main difficulty is that $Z^\epsilon(sT_\epsilon^1)$, as a process, takes values in $G_\epsilon$ rather than in $\mathcal{S}$.
The next lemma shows that if the process does not escape to another vertex $j \neq j_1$, then it is very likely to remain in a small neighborhood of $O_{j_1}$.
Together with \eqref{dis exit time upper}, this implies that the process spends most of its time near the vertices, and consequently behaves like $Y(s)$.

Before stating the lemma, we first recall that
\begin{equation*}
\begin{aligned}
    G_{\epsilon,j}(\delta) &\coloneqq  \{ z \in G_\epsilon : d_{\Gamma}(\Pi(z),O_j) \geq \delta \}, \quad
    B_{\epsilon,j}(\delta) \coloneqq  \{ z \in G_\epsilon : d_{\Gamma}(\Pi(z),O_j) \leq \delta \}.
\end{aligned}
\end{equation*}

\begin{lemma}\label{conti lem}
For any $s > 0$ and $\delta \in \left(0,\frac{1}{4}\min_{k=1}^{\abs{E}} \abs{I_k} \right)$, if $j=j_1$, we have
    \begin{equation}
\begin{aligned}
    & \lim_{\epsilon \to 0} \sup_{z \in C_{\epsilon.j}(r_j(\epsilon)+3\epsilon)} \mathbb{P}_z \left(     Z^\epsilon( sT_\epsilon^1)  \in G_{\epsilon,j}(\delta),  sT_\epsilon^1 \leq \sigma_j^\epsilon(\vec{L}_j^\epsilon) \right) = 0.
\end{aligned}
\end{equation}
\end{lemma}
\begin{proof}
    Recall that the sequence of stopping times
\begin{equation*}
\begin{aligned}
\sigma_{j,n}^{\epsilon,\delta}(\vec{L}_j^\epsilon) &\coloneqq \inf  \{t \geq  \tau_{j,n}^{\epsilon,\delta}(\vec{L}_j^\epsilon) : Z^\epsilon(t) \in  C_{\epsilon,j}(\vec{L}_j^\epsilon) \cup C_{\epsilon,j}(r_j(\epsilon)+ 3\epsilon ) \}, \\
\tau_{j,n}^{\epsilon,\delta}(\vec{L}_j^\epsilon) &\coloneqq \inf  \{t \geq  \sigma_{j,n-1}^{\epsilon,\delta}(\vec{L}_j^\epsilon) : Z^\epsilon(t) \in C_{\epsilon,j}(\delta) \},
\end{aligned}
\end{equation*}
with $\tau_{j,0}^{\epsilon,\delta}(\vec{L}_j^\epsilon) \coloneqq 0$.
For every $z \in C_{\epsilon.j}(r_j(\epsilon)+3\epsilon)$, we have
\begin{equation}
\begin{aligned}
    &\mathbb{P}_z \left(     Z^\epsilon( sT_\epsilon^1)  \in G_{\epsilon,j}(\delta),  sT_\epsilon^1 \leq \sigma_j^\epsilon(\vec{L}_j^\epsilon) \right) \\
    &\leq \sum_{n=0}^\infty  \mathbb{P}_z \left( \tau_{j,n}^{\epsilon,\delta}(\vec{L}_j^\epsilon) \leq sT_\epsilon^1  \leq  \sigma_{j,n}^{\epsilon,\delta}(\vec{L}_j^\epsilon), sT_\epsilon^1  \leq  \sigma_j^\epsilon(\vec{L}_j^\epsilon)  \right).
\end{aligned}
\end{equation}
It suffices to show that  
\begin{equation}\label{cont cla}
    \lim_{\epsilon \to 0}  \sum_{n=0}^\infty \sup_{z \in C_{\epsilon.j}(r_j(\epsilon)+3\epsilon)} \mathbb{P}_z \left( \tau_{j,n}^{\epsilon,\delta}(\vec{L}_j^\epsilon) \leq sT_\epsilon^1  \leq  \sigma_{j,n}^{\epsilon,\delta}(\vec{L}_j^\epsilon), sT_\epsilon^1  \leq  \sigma_j^\epsilon(\vec{L}_j^\epsilon)  \right) = 0.
\end{equation}

First, for all $n \geq 2$, by the strong Markov property,
\begin{equation}
\begin{aligned}
    & \sup_{z \in C_{\epsilon.j}(r_j(\epsilon)+3\epsilon)} \mathbb{P}_z \left( \tau_{j,n}^{\epsilon,\delta}(\vec{L}_j^\epsilon) \leq sT_\epsilon^1   \leq  \sigma_{j,n}^{\epsilon,\delta}(\vec{L}_j^\epsilon), sT_\epsilon^1  \leq  \sigma_j^\epsilon(\vec{L}_j^\epsilon)  \right) \\
    & \leq \sup_{z \in C_{\epsilon.j}(r_j(\epsilon)+3\epsilon)} \mathbb{P}_z \left( \tau_{j,n-1}^{\epsilon,\delta}(\vec{L}_j^\epsilon) \leq \sigma_{j,n-1}^{\epsilon,\delta}(\vec{L}_j^\epsilon) <  \sigma_j^\epsilon(\vec{L}_j^\epsilon)  \right) \\
    & \leq \sup_{z \in C_{\epsilon.j}(r_j(\epsilon)+3\epsilon)} \mathbb{E}_z \left(  \mathbbm{1}_{ \{ \tau_{j,n-2}^{\epsilon,\delta}(\vec{L}_j^\epsilon) \leq \sigma_{j,n-2}^{\epsilon,\delta}(\vec{L}_j^\epsilon) <  \sigma_j^\epsilon(\vec{L}_j^\epsilon) \} } 
    \mathbbm{1}_{ \{ \tau_{j,n-1}^{\epsilon,\delta}(\vec{L}_j^\epsilon) \leq \sigma_{j,n-1}^{\epsilon,\delta}(\vec{L}_j^\epsilon) <  \sigma_j^\epsilon(\vec{L}_j^\epsilon) \} }\right) \\
    & \leq \sup_{z' \in C_{\epsilon,j}(\delta) } \mathbb{P}_{z'} \left( \sigma_{j,0}^{\epsilon,\delta }(\vec{L}_j^\epsilon) < \sigma_j^\epsilon(\vec{L}_j^\epsilon) \right)  \\
    & \ \ \  \times \sup_{z \in C_{\epsilon.j}(r_j(\epsilon)+3\epsilon)} \mathbb{E}_z \left(  \mathbbm{1}_{ \{ \tau_{j,n-2}^{\epsilon,\delta}(\vec{L}_j^\epsilon) \leq \sigma_{j,n-2}^{\epsilon,\delta}(\vec{L}_j^\epsilon) <  \sigma_j^\epsilon(\vec{L}_j^\epsilon) \} }  \right)\\
     &\leq \left( \sup_{z' \in C_{\epsilon,j}(\delta) } \mathbb{P}_{z'} \left( \sigma_{j,0}^{\epsilon,\delta }(\vec{L}_j^\epsilon) < \sigma_j^\epsilon(\vec{L}_j^\epsilon) \right)  \right)^{n-1}
     \leq \left(1-\frac{\delta}{2\max_{k=1}^{\abs{E}} \abs{I_k}} \right)^{n-1}.
\end{aligned}
\end{equation}
As a result, for all $\eta > 0$, there exists $N_\eta >0$ such that
\begin{equation}\label{cont lim 1}
\begin{aligned}
    &\sum_{n=N_\eta+1}^\infty \sup_{z \in C_{\epsilon.j}(r_j(\epsilon)+3\epsilon)}  \mathbb{P}_z \left( \tau_{j,n}^{\epsilon,\delta}(\vec{L}_j^\epsilon) \leq sT_\epsilon^1  \leq  \sigma_{j,n}^{\epsilon,\delta}(\vec{L}_j^\epsilon), sT_\epsilon^1  \leq  \sigma_j^\epsilon(\vec{L}_j^\epsilon)  \right) \\
    &\leq \sum_{n=N_\eta+1}^\infty \left(1-\frac{\delta}{2\max_{k=1}^{\abs{E}} \abs{I_k}}  \right)^{n-1}
    \leq \frac{2\max_{k=1}^{\abs{E}} \abs{I_k}}{\delta} \left(1-\frac{\delta}{2\max_{k=1}^{\abs{E}} \abs{I_k}} \right)^{N_\eta} <\eta.
\end{aligned}
\end{equation}

Next, for every $n \geq 1$ and $\delta' >0$,
\begin{equation}
\begin{aligned}
    & \mathbb{P}_z \left( \tau_{j,n}^{\epsilon,\delta}(\vec{L}_j^\epsilon) \leq sT_\epsilon^1  \leq  \sigma_{j,n}^{\epsilon,\delta}(\vec{L}_j^\epsilon), sT_\epsilon^1  \leq  \sigma_j^\epsilon(\vec{L}_j^\epsilon)  \right) \\
    & \leq \mathbb{P}_z \left( \sigma_{j,n-1}^{\epsilon,\delta}(\vec{L}_j^\epsilon) \leq (s-\delta')T_\epsilon^1  \leq  \tau_{j,n}^{\epsilon,\delta}(\vec{L}_j^\epsilon) \leq sT_\epsilon^1  \leq  \sigma_{j,n}^{\epsilon,\delta}(\vec{L}_j^\epsilon), sT_\epsilon^1  \leq  \sigma_j^\epsilon(\vec{L}_j^\epsilon)  \right)  \\
    & \ \ \ + \mathbb{P}_z \left( (s-\delta')T_\epsilon^1 \leq  \sigma_{j,n-1}^{\epsilon,\delta}(\vec{L}_j^\epsilon) \leq  \tau_{j,n}^{\epsilon,\delta}(\vec{L}_j^\epsilon) \leq sT_\epsilon^1  \leq  \sigma_{j,n}^{\epsilon,\delta}(\vec{L}_j^\epsilon), sT_\epsilon^1  \leq  \sigma_j^\epsilon(\vec{L}_j^\epsilon)  \right)  \\
    & \ \ \ \ \ + \mathbb{P}_z \left( \tau_{j,n}^{\epsilon,\delta}(\vec{L}_j^\epsilon) \leq (s-\delta')T_\epsilon^1 \leq sT_\epsilon^1  \leq  \sigma_{j,n}^{\epsilon,\delta}(\vec{L}_j^\epsilon), sT_\epsilon^1  \leq  \sigma_j^\epsilon(\vec{L}_j^\epsilon)  \right) \\
    & \eqqcolon I_{1}^{\epsilon,\delta'}(z) + I_{2}^{\epsilon,\delta'}(z) + I_{3}^{\epsilon,\delta'}(z).
\end{aligned}
\end{equation}

For $I_{1}^{\epsilon,\delta'}$, first note that
\begin{equation}
\begin{aligned}
    &I_{1}^{\epsilon,\delta'}(z) \leq \mathbb{E}_z \Big( \mathbbm{1}_{ \{ \sigma_{j,n-1}^{\epsilon,\delta}(\vec{L}_j^\epsilon) \leq (s-\delta')T_\epsilon^1   \} } \mathbbm{1}_{ \{ \sigma_{j,n-1}^{\epsilon,\delta}(\vec{L}_j^\epsilon) <  \sigma_j^\epsilon(\vec{L}_j^\epsilon) \} } \\
    & \ \ \ \ \ \ \ \ \ \ \ \ \ \times \mathbb{P}_{Z^\epsilon(\sigma_{j,n-1}^{\epsilon,\delta}(\vec{L}_j^\epsilon))} \left( \sigma_{j}^{\epsilon}(\delta) \in [(s-\delta')T_\epsilon^1-\sigma_{j,n-1}^{\epsilon,\delta}(\vec{L}_j^\epsilon), sT_\epsilon^1-\sigma_{j,n-1}^{\epsilon,\delta}(\vec{L}_j^\epsilon)] \right) \Big).
\end{aligned}
\end{equation}
Under the event $\{ \sigma_{j,n-1}^{\epsilon,\delta}(\vec{L}_j^\epsilon) \leq (s-\delta')T_\epsilon^1  \} \cap \{ \sigma_{j,n-1}^{\epsilon,\delta}(\vec{L}_j^\epsilon)  <  \sigma_j^\epsilon(\vec{L}_j^\epsilon) \}$, we have
\begin{equation}
    \begin{aligned}
        &\mathbb{P}_{Z^\epsilon(\sigma_{j,n-1}^{\epsilon,\delta}(\vec{L}_j^\epsilon))} \left( \sigma_{j}^{\epsilon}(\delta) \in [(s-\delta')T_\epsilon^1-\sigma_{j,n-1}^{\epsilon,\delta}(\vec{L}_j^\epsilon), sT_\epsilon^1-\sigma_{j,n-1}^{\epsilon,\delta}(\vec{L}_j^\epsilon)] \right) \\
        &\leq \sup_{z \in C_{\epsilon,j}(r_{j}(\epsilon)+3\epsilon)}  \mathbb{P}_z \left( \sigma_{j}^{\epsilon}(\delta) \in [(s-\delta')T_\epsilon^1-\sigma_{j,n-1}^{\epsilon,\delta}(\vec{L}_j^\epsilon), sT_\epsilon^1-\sigma_{j,n-1}^{\epsilon,\delta}(\vec{L}_j^\epsilon)] \right) \\
        &= \sup_{z \in C_{\epsilon,j}(r_{j}(\epsilon)+3\epsilon)} \mathbb{P}_z \left( \sigma_{j}^{\epsilon}(\delta) \geq (s-\delta')T_\epsilon^1-\sigma_{j,n-1}^{\epsilon,\delta}(\vec{L}_j^\epsilon) \right) \\
        & \ \ \ -\mathbb{P}_z \left( \sigma_{j}^{\epsilon}(\delta) \geq sT_\epsilon^1-\sigma_{j,n-1}^{\epsilon,\delta}(\vec{L}_j^\epsilon) \right) .
    \end{aligned}
\end{equation}
Thanks to \eqref{exp est 1}, for every $\eta >0$ and $\delta' > 0$, there exists $\epsilon_{1,\eta,\delta'} > 0$ such that for every $\epsilon \in (0, \epsilon_{1,\eta,\delta'})$
\begin{equation}
\begin{aligned}
    &\sup_{z \in C_{\epsilon,j}(r_{j}(\epsilon)+3\epsilon)} \mathbb{P}_z \left( \sigma_{j}^{\epsilon}(\delta) \geq (s-\delta')T_\epsilon^1-\sigma_{j,n-1}^{\epsilon,\delta}(\vec{L}_j^\epsilon) \right)
        -\mathbb{P}_z \left( \sigma_{j}^{\epsilon}(\delta) \geq sT_\epsilon^1-\sigma_{j,n-1}^{\epsilon,\delta}(\vec{L}_j^\epsilon) \right)  \\
        &\leq e^{-(s-\delta'-\sigma_{j,n-1}^{\epsilon,\delta}/T_\epsilon^1)\kappa_{j}(\delta)}+\frac{\eta}{9}
        -e^{-(s-\sigma_{j,n-1}^{\epsilon,\delta}/T_\epsilon^1)\kappa_{j}(\delta)}+\frac{\eta}{9} \\
        &\leq \delta' \kappa_{j}(\delta) + \frac{2\eta}{9}.
\end{aligned}
\end{equation}
Consequently, for every $\eta > 0$ there exists $ \delta'_{1,\eta} >0$ such that for all $\delta' \in (0,\delta'_{1,\eta})$, there exists $\epsilon_{1,\eta,\delta'}>0$ such that 
\begin{equation}\label{conti est 1}
\begin{aligned}
    & \sup_{z \in C_{\epsilon,j}(r_j(\epsilon)+3\epsilon)} I_{1}^{\epsilon,\delta'}(z) \\
    &\leq \sup_{z \in C_{\epsilon,j}(r_j(\epsilon)+3\epsilon)} \mathbb{E}_z \Big( \mathbbm{1}_{ \{ \sigma_{j,n-1}^{\epsilon,\delta}(\vec{L}_j^\epsilon) \leq (s-\delta')T_\epsilon^1   \} } \mathbbm{1}_{ \{ \sigma_{j,n-1}^{\epsilon,\delta}(\vec{L}_j^\epsilon) <  \sigma_j^\epsilon(\vec{L}_j^\epsilon) \} } \\
    & \ \ \ \ \ \ \  \times \mathbb{P}_{Z^\epsilon(\sigma_{j,n-1}^{\epsilon,\delta}(\vec{L}_j^\epsilon))} \left( \sigma_{j}^{\epsilon}(\delta) \in [(s-\delta')T_\epsilon^1-\sigma_{j,n-1}^{\epsilon,\delta}(\vec{L}_j^\epsilon), sT_\epsilon^1-\sigma_{j,n-1}^{\epsilon,\delta}(\vec{L}_j^\epsilon)] \right) \Big) \\
    & \leq \frac{\eta}{3},
\end{aligned}
\end{equation}
for every $\epsilon \in (0,\epsilon_{1,\eta,\delta'})$.

For $I_{2}^{\epsilon,\delta'}$, by the strong Markov property,
\begin{equation}
\begin{aligned}
    &\mathbb{P}_z \left( (s-\delta')T_\epsilon^1 \leq  \sigma_{j,n-1}^{\epsilon,\delta}(\vec{L}_j^\epsilon) \leq  \tau_{j,n}^{\epsilon,\delta}(\vec{L}_j^\epsilon) \leq sT_\epsilon^1  \leq  \sigma_{j,n}^{\epsilon,\delta}(\vec{L}_j^\epsilon), sT_\epsilon^1  \leq  \sigma_j^\epsilon(\vec{L}_j^\epsilon)  \right) \\
    & \leq \mathbb{E}_z \left( \mathbbm{1}_{ \{ \sigma_{j,n-1}^{\epsilon,\delta}(\vec{L}_j^\epsilon) < \sigma_j^\epsilon(\vec{L}_j^\epsilon) \} } \mathbb{P}_{Z^\epsilon ( \sigma_{j,n-1}^{\epsilon,\delta}(\vec{L}_j^\epsilon) )} \left( \sigma_j^\epsilon(\delta) <\delta'T_\epsilon^1 \right) \right) \\
    & \leq \sup_{z' \in C_{\epsilon,j}(r_j(\epsilon)+3\epsilon)} \mathbb{P}_{z'} \left( \sigma_j^\epsilon(\delta) <\delta'T_\epsilon^1 \right) 
    = 1-\inf_{z' \in C_{\epsilon,j}(r_j(\epsilon)+3\epsilon)} \mathbb{P}_{z'} \left( \sigma_j^\epsilon(\delta) \geq \delta'T_\epsilon^1 \right) .
\end{aligned}
\end{equation}
Due to \eqref{exp est 1}, for every $\eta > 0$ there exists $ \delta'_{2,\eta} \in (0,\delta'_{1,\eta})$ such that for all $\delta' \in (0,\delta'_{2,\eta})$, there exists $\epsilon_{2,\eta,\delta'} \in (0,\epsilon_{2,\eta,\delta'})$ such that 
\begin{equation}\label{conti est 2}
\begin{aligned}
    &\sup_{z \in C_{\epsilon,j}(r_j(\epsilon)+3\epsilon)} I_{2}^{\epsilon,\delta'}(z) \leq 1-\inf_{z' \in C_{\epsilon,j}(r_j(\epsilon)+3\epsilon)} \mathbb{P}_{z'} \left( \sigma_j^\epsilon(\delta) \geq \delta'T_\epsilon^1 \right) \\
    & \leq 1-e^{-\delta' \kappa_{j}(\delta)}+\frac{\eta}{6}\leq \frac{\eta}{3},
\end{aligned}
\end{equation}
for every $\epsilon \in (0,\epsilon_{2,\eta,\delta'})$.

For $I_{3}^{\epsilon,\delta'}$, by the strong Markov property,
\begin{equation}
\begin{aligned}
    &\mathbb{P}_z \left( \tau_{j,n}^{\epsilon,\delta}(\vec{L}_j^\epsilon) \leq (s-\delta')T_\epsilon^1 \leq sT_\epsilon^1  \leq  \sigma_{j,n}^{\epsilon,\delta}(\vec{L}_j^\epsilon), sT_\epsilon^1  \leq  \sigma_j^\epsilon(\vec{L}_j^\epsilon)  \right) \\
    &\leq \mathbb{E}_z \left( \mathbbm{1}_{ \{ \tau_{j,n}^{\epsilon,\delta}(\vec{L}_j^\epsilon)  < \sigma_j^\epsilon(\vec{L}_j^\epsilon) \} } \mathbb{P}_{Z^\epsilon(\tau_{j,n}^{\epsilon,\delta}(\vec{L}_j^\epsilon))} \left( \sigma_{j,0}^{\epsilon,\delta}(\vec{L}_j^\epsilon) >\delta' T_\epsilon^1\right) \right) \\
    & \leq \sup_{z' \in C_{\epsilon,j}(\delta)} \frac{\mathbb{E}_{z'} \sigma_{j,0}^{\epsilon,\delta}(\vec{L}_j^\epsilon) }{\delta' T_\epsilon^1} \leq c \frac{\delta}{\delta' T_\epsilon^1}.
\end{aligned}
\end{equation}
As a result, for every $\eta >0$ and $\delta' >0$, there exists $\epsilon_{3,\eta,\delta'} \in (0,\epsilon_{2,\eta,\delta'})$ such that for every $\epsilon \in (0,\epsilon_{3,\eta,\delta'})$, we have
\begin{equation}\label{conti est 3}
    \sup_{z \in C_{\epsilon,j}(r_j(\epsilon)+3\epsilon)} I_{3}^{\epsilon,\delta'}(z) < \frac{\eta}{3}.
\end{equation}

Combining \eqref{conti est 1}, \eqref{conti est 2}, \eqref{conti est 3}, we have for every $n \geq 1$,
\begin{equation}\label{cont lim 2}
    \lim_{\epsilon \to 0} \sup_{z \in C_{\epsilon.j}(r_j(\epsilon)+3\epsilon)}  \mathbb{P}_z \left( \tau_{j,n}^{\epsilon,\delta}(\vec{L}_j^\epsilon) \leq sT_\epsilon^1  \leq  \sigma_{j,n}^{\epsilon,\delta}(\vec{L}_j^\epsilon), sT_\epsilon^1  \leq  \sigma_j^\epsilon(\vec{L}_j^\epsilon)  \right) = 0.
\end{equation}
Moreover, thanks to \eqref{cont lim 1} and \eqref{cont lim 2}, we can conclude that
\begin{equation}
\begin{aligned}
    &\lim_{\epsilon \to 0} \sum_{n=0}^\infty \sup_{z \in C_{\epsilon.j}(r_j(\epsilon)+3\epsilon)}  \mathbb{P}_z \left( \tau_{j,n}^{\epsilon,\delta}(\vec{L}_j^\epsilon) \leq sT_\epsilon^1  \leq  \sigma_{j,n}^{\epsilon,\delta}(\vec{L}_j^\epsilon), sT_\epsilon^1  \leq  \sigma_j^\epsilon(\vec{L}_j^\epsilon)  \right) \\
    &= \lim_{\epsilon \to 0} \sum_{n=0}^{N_\eta} \sup_{z \in C_{\epsilon.j}(r_j(\epsilon)+3\epsilon)}  \mathbb{P}_z \left( \tau_{j,n}^{\epsilon,\delta}(\vec{L}_j^\epsilon) \leq sT_\epsilon^1  \leq  \sigma_{j,n}^{\epsilon,\delta}(\vec{L}_j^\epsilon), sT_\epsilon^1  \leq  \sigma_j^\epsilon(\vec{L}_j^\epsilon)  \right)
    + \eta < \eta.
\end{aligned}
\end{equation}
Since $\eta > 0$ is arbitrary, we have \eqref{cont cla}.

\end{proof}

We have the following theorem
\begin{theorem}
	For any $s > 0$ and any continuous function $F \in C(\Gamma)$, we have
	\begin{equation}
		\lim_{\epsilon \to 0} \sup_{z \in C_{\epsilon,j}( r_j(\epsilon)+3\epsilon)} \abs{ \mathbb{E}_z F(\Pi(Z^\epsilon(sT_\epsilon^1))) - \mathbf{E}_{O_j} F(Y(s)) } = 0.
	\end{equation}
\end{theorem}
\begin{proof}
First, by \eqref{dis exit time upper} and the definition of $Y(s)$, it suffices to consider the case $j=j_1$.
We have
\begin{equation}
\begin{aligned}
    &  \mathbb{E}_z F(\Pi(Z^\epsilon(sT_\epsilon^1))) 
     =   \mathbb{E}_z F(\Pi(Z^\epsilon(sT_\epsilon^1)))  \mathbbm{1}_{ \{ \sigma_j^\epsilon(\vec{L}_j^\epsilon) \geq sT_\epsilon^1 \} } \\
     & \ \ \ +   \mathbb{E}_z F(\Pi(Z^\epsilon(sT_\epsilon^1)))  \mathbbm{1}_{ \{ \sigma_j^\epsilon(\vec{L}_j^\epsilon) < sT_\epsilon^1 \} } \\
    &\eqqcolon I_1^\epsilon(z) + I_2^\epsilon(z).
\end{aligned}
\end{equation}
For $I_2^\epsilon$,
\begin{equation}
\begin{aligned}
    &\mathbb{E}_z F(\Pi(Z^\epsilon(sT_\epsilon^1)))  \mathbbm{1}_{ \{ \sigma_j^\epsilon(\vec{L}_j^\epsilon) < sT_\epsilon^1 \} } \\
    &= \mathbb{E}_z \left[ \mathbbm{1}_{ \{ \sigma_j^\epsilon(\vec{L}_j^\epsilon) < sT_\epsilon^1 \} } \mathbb{E}_{\sigma_j^\epsilon(\vec{L}_j^\epsilon)} F(\Pi(Z^\epsilon(sT_\epsilon^1- \sigma_j^\epsilon(\vec{L}_j^\epsilon))))  
    \right] \\
    & =  \mathbb{E}_z \left[ \mathbbm{1}_{ \{ \sigma_j^\epsilon(\vec{L}_j^\epsilon) < sT_\epsilon^1 \} } F(\Pi(Z^\epsilon(\sigma_j^\epsilon(\vec{L}_j^\epsilon))))  
    \right] + J_{2,1}^\epsilon(z), 
\end{aligned}
\end{equation}
where
\begin{equation}
\begin{aligned}
    J_{2,1}^\epsilon(z) &\coloneqq \mathbb{E}_z \left[ \mathbbm{1}_{ \{ \sigma_j^\epsilon(\vec{L}_j^\epsilon) < sT_\epsilon^1 \} } \mathbb{E}_{\sigma_j^\epsilon(\vec{L}_j^\epsilon)} \left( F(\Pi(Z^\epsilon(sT_\epsilon^1- \sigma_j^\epsilon(\vec{L}_j^\epsilon))))  
    -F(\Pi(Z^\epsilon(0))) \right)
    \right] .
\end{aligned}
\end{equation}
Note that, under the events $\{ \sigma_j^\epsilon(\vec{L}_j^\epsilon) < sT_\epsilon^1 \}$ and $\{ Z^\epsilon(\sigma_j^\epsilon(\vec{L}_j^\epsilon)) \in C_{\epsilon,j'}(r_{j'}(\epsilon)+3\epsilon )\}$ for some $j' \neq j_1$, we have
\begin{equation}
\begin{aligned}
    &\abs{\mathbb{E}_{\sigma_j^\epsilon(\vec{L}_j^\epsilon)}  \left( F(\Pi(Z^\epsilon(sT_\epsilon^1- \sigma_j^\epsilon(\vec{L}_j^\epsilon))))  
    -F(\Pi(Z^\epsilon(0))) \right) } \\
    &= \abs{\mathbb{E}_{\sigma_j^\epsilon(\vec{L}_j^\epsilon)}  \left( F(\Pi(Z^\epsilon(sT_\epsilon^1- \sigma_j^\epsilon(\vec{L}_j^\epsilon))))  
    -F(\Pi(Z^\epsilon(0))) \right) \mathbbm{1}_{ \{ \sigma_{j'}^\epsilon(\delta) \leq sT_\epsilon^1 \} }  } \\
    & \ \ \ + \abs{\mathbb{E}_{\sigma_j^\epsilon(\vec{L}_j^\epsilon)}  \left( F(\Pi(Z^\epsilon(sT_\epsilon^1- \sigma_j^\epsilon(\vec{L}_j^\epsilon))))  
    -F(\Pi(Z^\epsilon(0))) \right) \mathbbm{1}_{ \{ \sigma_{j'}^\epsilon(\delta) > sT_\epsilon^1 \} } } \\
    & \leq 2\abs{F}_{C(\Gamma)} \max_{j' \neq j_1} \sup_{z \in C_{\epsilon,j'}(r_{j'}(\epsilon)+3\epsilon)} \mathbb{P}_z \left(  \sigma_{j'}^\epsilon(\delta) \leq sT_\epsilon^1 \right) +o_{\epsilon,\delta}(1) = o(1)+ o_{\epsilon,\delta}(1).
\end{aligned}
\end{equation}
The $o(1)$ term is due to the fact that since all $j' \neq j_1$, we have $r_j(\epsilon) \gg r_{(1)}(\epsilon)$, the result follows from \eqref{dis exit time upper}.
The $o_{\epsilon,\delta}(1)$ is by the continuity of $F$.
We then conclude that
\begin{equation}
\begin{aligned}
   &\sup_{z \in C_{\epsilon,j}( r_{j}(\epsilon)+3\epsilon)} J_{2,1}^\epsilon (z) \\
   &\leq 2\abs{F}_{C(\Gamma)} \max_{j' \neq j_1} \sup_{z \in C_{\epsilon,j'}(r_{j'}(\epsilon)+3\epsilon)} \mathbb{P}_z \left(  \sigma_{j'}^\epsilon(\delta) \leq sT_\epsilon^1 \right) +o_{\epsilon,\delta}(1) = o_{\epsilon,\delta}(1) .
\end{aligned}
\end{equation}

Now for $I_1^\epsilon$, 
\begin{equation}
\begin{aligned}
    &\mathbb{E}_z F(\Pi(Z^\epsilon(sT_\epsilon^1)))  \mathbbm{1}_{ \{ \sigma_j^\epsilon(\vec{L}_j^\epsilon) \geq sT_\epsilon^1 \} } \\
    &=  \mathbb{E}_z F(\Pi(Z^\epsilon(sT_\epsilon^1)))  \mathbbm{1}_{ \{  \sigma_j^\epsilon(\vec{L}_j^\epsilon) \geq sT_\epsilon^1   \} }  \mathbbm{1}_{ \{ Z^\epsilon( sT_\epsilon^1)  \in B_{\epsilon,j}(\delta)  \} } \\
    & \ \ \ +  \mathbb{E}_z F(\Pi(Z^\epsilon(sT_\epsilon^1)))  \mathbbm{1}_{ \{  \sigma_j^\epsilon(\vec{L}_j^\epsilon) \geq sT_\epsilon^1   \} }  \mathbbm{1}_{ \{ Z^\epsilon( sT_\epsilon^1)  \in G_{\epsilon,j}(\delta)  \} } .
\end{aligned}
\end{equation}
Due to Lemma \ref{conti lem},
\begin{equation}
\begin{aligned}
    &\sup_{z \in C_{\epsilon,j}( r_{j}(\epsilon)+3\epsilon)}  \abs{\mathbb{E}_z F(\Pi(Z^\epsilon(sT_\epsilon^1)))  \mathbbm{1}_{ \{  \sigma_j^\epsilon(\vec{L}_j^\epsilon) \geq sT_\epsilon^1   \} }  \mathbbm{1}_{ \{ Z^\epsilon( sT_\epsilon^1)  \in G_{\epsilon,j}(\delta)  \} }} \\
    &\leq  \abs{F}_{C(\Gamma)} \sup_{z \in C_{\epsilon,j}( r_{j}(\epsilon)+3\epsilon)}   \mathbb{P}_z \left( Z^\epsilon( sT_\epsilon^1)  \in G_{\epsilon,j}(\delta), \sigma_j^\epsilon(\vec{L}_j^\epsilon) \geq sT_\epsilon^1 \right)
    =o(1).
\end{aligned}
\end{equation}

Finally,
\begin{equation}
\begin{aligned}
    &\sup_{z \in C_{\epsilon,j}( r_j(\epsilon)+3\epsilon)} \abs{ \mathbb{E}_z F(\Pi(Z^\epsilon(sT_\epsilon^1))) - \mathbf{E}_{O_j} F(Y(s)) } \\
    &\leq \sup_{z \in C_{\epsilon,j}( r_j(\epsilon)+3\epsilon)} \Big|\mathbb{E}_z F(\Pi(Z^\epsilon(sT_\epsilon^1)))  \mathbbm{1}_{ \{  \sigma_j^\epsilon(\vec{L}_j^\epsilon) \geq sT_\epsilon^1   \} }  \mathbbm{1}_{ \{ Z^\epsilon( sT_\epsilon^1)  \in B_{\epsilon,j}(\delta)  \} } \\
    & \ \ \ \ \ \  \ \ \ \ \ \ \ \  \ \ \ \ \ \ \ \ \  \  + \mathbb{E}_z   \mathbbm{1}_{ \{ \sigma_j^\epsilon(\vec{L}_j^\epsilon) < sT_\epsilon^1 \} } F(\Pi(Z^\epsilon(\sigma_j^\epsilon(\vec{L}_j^\epsilon))))  
    -  \mathbf{E}_{O_j} F(Y(s))\Big| + o_{\epsilon,\delta}(1) \\
    & \leq \sup_{z \in C_{\epsilon,j}( r_j(\epsilon)+3\epsilon)} \bigg| \mathbb{E}_z F(\Pi(Z^\epsilon(sT_\epsilon^1)))  \mathbbm{1}_{ \{  \sigma_j^\epsilon(\vec{L}_j^\epsilon) \geq sT_\epsilon^1   \} }  \mathbbm{1}_{ \{ Z^\epsilon( sT_\epsilon^1)  \in B_{\epsilon,j}(\delta)  \} } \\
    & \ \ \ \ \ \  \ \ \ \ \ \ \ \  \ \ \ \ \ \ \ \ \  \ \ \ -  \mathbf{E}_{O_j} F(Y(s)) \mathbbm{1}_{ \{ \sigma_1 \geq s \} } \bigg| \\
    & \ \ \ \ \ + \sup_{z \in C_{\epsilon,j}( r_j(\epsilon)+3\epsilon)} \bigg| \mathbb{E}_z \mathbbm{1}_{ \{ \sigma_j^\epsilon(\vec{L}_j^\epsilon) < sT_\epsilon^1 \} } F(\Pi(Z^\epsilon(\sigma_j^\epsilon(\vec{L}_j^\epsilon))))  \\
    & \ \ \ \ \ \  \ \ \ \ \ \ \ \  \ \ \ \ \ \ \ \ \  \ \ \ \ \ \ \ \ -  \mathbf{E}_{O_j} F(Y(s)) \mathbbm{1}_{ \{ \sigma_1 < s \} } \bigg| + o_{\epsilon,\delta}(1) \\
    & \eqqcolon I_3^\epsilon + I_4^\epsilon + o_{\epsilon,\delta}(1).
\end{aligned}
\end{equation}
For $I_3^\epsilon$, by the continuity of $F$, Lemma \ref{conti lem}, equation \eqref{exp est 1} and the definition of $\sigma_1$, we have
\begin{equation}
\begin{aligned}
    &\sup_{z \in C_{\epsilon,j}( r_j(\epsilon)+3\epsilon)} \bigg| \mathbb{E}_z F(\Pi(Z^\epsilon(sT_\epsilon^1)))  \mathbbm{1}_{ \{  \sigma_j^\epsilon(\vec{L}_j^\epsilon) \geq sT_\epsilon^1   \} }  \mathbbm{1}_{ \{ Z^\epsilon( sT_\epsilon^1)  \in B_{\epsilon,j}(\delta)  \} } \\
    & \ \ \ \ \  \ \ \ \ \ \ \ \ \ \ \ \ \ \ \ \ \ \ \ \ \ \ -  \mathbf{E}_{O_j} F(Y(s)) \mathbbm{1}_{ \{ \sigma_1 \geq s \} } \bigg| \\
    & \leq \sup_{z \in C_{\epsilon,j}( r_j(\epsilon)+3\epsilon)}  \mathbb{E}_z \abs{F(\Pi(Z^\epsilon(sT_\epsilon^1))) - F(O_j)} \mathbbm{1}_{ \{  \sigma_j^\epsilon(\vec{L}_j^\epsilon) \geq sT_\epsilon^1   \} }  \mathbbm{1}_{ \{ Z^\epsilon( sT_\epsilon^1)  \in B_{\epsilon,j}(\delta)  \} } \\
    & \ \ \ \ \  + \abs{F}_{C(\Gamma)} \sup_{z \in C_{\epsilon,j}( r_j(\epsilon)+3\epsilon)}  \mathbb{P}_z  \left( \sigma_j^\epsilon(\vec{L}_j^\epsilon) \geq sT_\epsilon^1 ,  Z^\epsilon( sT_\epsilon^1)  \in G_{\epsilon,j}(\delta)  \right)   \\
    & \ \ \ \ \ \ \ \ \ \ + \abs{F}_{C(\Gamma)} \sup_{z \in C_{\epsilon,j}( r_j(\epsilon)+3\epsilon)} \abs{ \mathbb{P}_z  \left( \sigma_j^\epsilon(\vec{L}_j^\epsilon) \geq sT_\epsilon^1  \right) -  \mathbf{P}_{O_j} \left( \sigma_1 \geq s \right) } \\
    & = o_{\epsilon,\delta}(1) + o(1) + o(1).
\end{aligned}
\end{equation}
For $I_4^\epsilon$, by the continuity of $F$, equation \eqref{exp est 2}, Lemma \ref{exit place est new} and the definition of $Y(s)$, we have
\begin{equation}
\begin{aligned}
    &\sup_{z \in C_{\epsilon,j}( r_j(\epsilon)+3\epsilon)} \abs{ \mathbb{E}_z \mathbbm{1}_{ \{ \sigma_j^\epsilon(\vec{L}_j^\epsilon) < sT_\epsilon^1 \} } F(\Pi(Z^\epsilon(\sigma_j^\epsilon(\vec{L}_j^\epsilon))))  -  \mathbf{E}_{O_j} F(Y(s)) \mathbbm{1}_{ \{ \sigma_1 < s \} } } \\
    &\leq \sup_{z \in C_{\epsilon,j}( r_j(\epsilon)+3\epsilon)} \sum_{j'=1}^{\abs{V}} \Big|
    \mathbb{E}_z \mathbbm{1}_{ \{ \sigma_j^\epsilon(\vec{L}_j^\epsilon) < sT_\epsilon^1 \} } \mathbbm{1}_{ \{ Z^\epsilon(\sigma_j^\epsilon(\vec{L}_j^\epsilon)) \in C_{\epsilon,j'}(r_{j'}(\epsilon)+3\epsilon) \} }F(\Pi(Z^\epsilon(\sigma_j^\epsilon(\vec{L}_j^\epsilon)))) \\
    & \ \ \ \ \ \ \ \ \ \ \ \ \ \ \ \ \ \ \ \ \ \ \ \ \ \ \ \ \ \ \ \ \ \ \ \ \ \ \ \ \ \ \ \ - \mathbf{E}_{O_j} F(Y(s)) \mathbbm{1}_{ \{ \sigma_1 < s \} } \mathbbm{1}_{ \{ \xi_1 =O_{j'} \} }
    \Big| \\
     &\leq \sup_{z \in C_{\epsilon,j}( r_j(\epsilon)+3\epsilon)} \sum_{j'=1}^{\abs{V}} 
    \mathbb{E}_z \mathbbm{1}_{ \{ \sigma_j^\epsilon(\vec{L}_j^\epsilon) < sT_\epsilon^1 \} } \mathbbm{1}_{ \{ Z^\epsilon(\sigma_j^\epsilon(\vec{L}_j^\epsilon)) \in C_{\epsilon,j'}(r_{j'}(\epsilon)+3\epsilon) \} } \\
    & \ \ \ \ \  \ \ \ \ \ \ \ \ \ \ \ \ \ \ \ \ \ \ \ \ \ \ \ \ \ \times
    \abs{ F(\Pi(Z^\epsilon(\sigma_j^\epsilon(\vec{L}_j^\epsilon)))) -F(O_{j'})} \\
     &+ \abs{F}_{C(\Gamma)} \sup_{z \in C_{\epsilon,j}( r_j(\epsilon)+3\epsilon)} \sum_{j'=1}^{\abs{V}} \Big|
    \mathbb{E}_z \mathbbm{1}_{ \{ \sigma_j^\epsilon(\vec{L}_j^\epsilon) < sT_\epsilon^1 \} } \mathbbm{1}_{ \{ Z^\epsilon(\sigma_j^\epsilon(\vec{L}_j^\epsilon)) \in C_{\epsilon,j'}(r_{j'}(\epsilon)+3\epsilon) \} } \\
    & \ \ \ \ \  \ \ \ \ \ \ \ \ \ \ \ \ \ \ \ \ \ \ \ \ \ \ \ \ \ \ \  \ \ \ \ \ \ \ \ \ \ \ \ \ \ \ \ \ \ \ \ \ \ - \mathbf{E}_{O_j} \mathbbm{1}_{ \{ \sigma_1 < s \} } \mathbbm{1}_{ \{ \xi_1 =O_{j'} \} }
    \Big| \\
    & = o(1) + o(1).
\end{aligned}
\end{equation}
This completes the proof.

\end{proof}

\subsection{Neumann problems}
First, recall that if $x \in I_k$ with endpoints $O_{j},O_{j'}$, then
\begin{equation}
	p(x,O_{j}) = 
        \frac{d(x,O_{j'})}{\abs{I_k}}
\end{equation}
and, for $O_{j''} \neq O_{j},O_{j'}$, 
\begin{equation}
    p(x,O_{j''}) = 0.
\end{equation}
We have the following corollary
\begin{corollary}\label{coro conv pde fir}
    Let $\{ \varphi_\epsilon \}_{\epsilon \in (0,1)}$ be equicontinuous and equibounded.
    For any $s>0$, we have
    \begin{equation}
        \lim_{\epsilon \to 0} \sup_{x \in \Gamma} \sup_{z: \Pi(z)=x} \abs{
        \mathbb{E}_{z} \varphi_\epsilon (Z^\epsilon(sT_\epsilon^1)) 
        - \sum_{j=1}^{\abs{V}} p(x,O_j) \mathbf{E}_{O_{j}} \varphi_\epsilon (Y(s)) } =0.
    \end{equation}
\end{corollary}
The proof follows the same argument as in Section \ref{sec pde} and is therefore omitted.

\subsection{Remarks}\label{sec rem}
It is also of interest to study other critical time scales, as in \cite{LLS24+,LM25}.
At an intuitive level, one may use \eqref{exp est 2} to construct a continuous-time Markov chain with the following behavior:
\begin{itemize}
\item for balls smaller than the critical time scale: it makes instantaneous jumps according to the probabilities given in Lemma \ref{exit place est new};
\item for balls larger than the critical time scale: it remains there forever;
\item for balls at the critical time scale: it jumps at rates determined by \eqref{exp est 2}, with transition probabilities given by Lemma \ref{exit place est new}.
\end{itemize}
To make the notion of “instantaneous jumps” rigorous, one possible approach is to construct a process $\{Y^\epsilon(s)\}_{s \geq 0}$ based on \eqref{exp est 2} and Lemma \ref{exit place est new}.

Alternatively, for each time scale $T_\epsilon^i$, one can first reduce the state space so that the new process excludes the balls smaller than the critical time scale but still captures their effective transitions.
However, this reduction depends strongly on the underlying graph structure, making it difficult to discuss in full generality.

\section*{Acknowledgment}
The author thanks Leonid Koralov for many helpful discussions related to this work and to \cite{SMP}.
The author is also grateful to Sandra Cerrai and Chenglin Liu for insightful conversations.

%%For cite command type as \cite{1}; \cite{3,6} and \cite{2,4,6}. 
%%For refcite command type as Refs.~[\refcite{1}];   
%%[\refcite{1},\refcite{3}] and [\refcite{1}--\refcite{4}].

\end{document}